\documentclass{article}
\usepackage[utf8]{inputenc} 
\usepackage[T1]{fontenc}    
\usepackage{geometry}
\usepackage{hyperref}       
\usepackage{url}            
\usepackage{booktabs}       
\usepackage{amsfonts}       
\usepackage{nicefrac}       
\usepackage{microtype}      
\usepackage{xcolor}         
\usepackage{float}
\usepackage{array}
\usepackage{subfig}
\usepackage{amsmath}
\usepackage{amssymb}
\usepackage{mathtools}
\usepackage{algorithmic}
\usepackage{amsthm}
\usepackage{dsfont}

\makeatletter

\providecommand{\tabularnewline}{\\}
\floatstyle{ruled}
\newfloat{algorithm}{tbp}{loa}
\providecommand{\algorithmname}{Algorithm}
\floatname{algorithm}{\protect\algorithmname}

\theoremstyle{remark}

\theoremstyle{plain}
\newtheorem{thm}{\protect\theoremname}[section]
\theoremstyle{definition}
\newtheorem{defn}{\protect\definitionname}[section]
\theoremstyle{plain}
\newtheorem{lem}{\protect\lemmaname}[section]
\theoremstyle{plain}
\newtheorem{assumption}{\protect\assumptionname}
\theoremstyle{plain}
\newtheorem{fact}{\protect\factname}[section]

\makeatother

\providecommand{\assumptionname}{Assumption}
\providecommand{\definitionname}{Definition}
\providecommand{\factname}{Fact}
\providecommand{\lemmaname}{Lemma}
\providecommand{\theoremname}{Theorem}
\providecommand{\remarkname}{Remark}
\newtheorem{invariant}{Invariant}

\begin{document}
\global\long\def\cE{{\cal E}}%
\global\long\def\opt{{\cal OPT}}%
\global\long\def\res{\mathrm{res}}%
\global\long\def\ressolver{\mathrm{ResidualSolver}}%
\global\long\def\subsolver{\mathrm{SubSolver}}%
\global\long\def\R{\mathbb{\mathbb{R}}}%
\global\long\def\one{\mathds{1}}%
\global\long\def\diag{\mathrm{diag}}%
\renewcommand{\epsilon}{\varepsilon}

\author{Alina Ene\thanks{
Department of Computer Science,Boston University, \texttt{aene@bu.edu}}
 \and
 Ta Duy Nguyen\thanks{Department of Computer Science, Boston University, \texttt{taduy@bu.edu}} 
 \and
Adrian Vladu\thanks{CNRS \& IRIF, Universit\'e Paris Cit\'e, \texttt{vladu@irif.fr}}
}
\title{Quasi-Self-Concordant Optimization with  Lewis Weights}
\date{}
\maketitle
\begin{abstract}
In this paper, we study the problem $\min_{x\in\R^{d},Nx=v}\sum_{i=1}^{n}f((Ax-b)_{i})$
for a quasi-self-concordant function $f:\R\to\R$, where $A,N$ are
$n\times d$ and $m\times d$ matrices, $b,v$ are vectors of length
$n$ and $m$ with $n\ge d.$ We show an algorithm based on a trust-region
method with an oracle that can be implemented using $\widetilde{O}(d^{1/3})$
linear system solves, improving the $\widetilde{O}(n^{1/3})$ oracle
by {[}Adil-Bullins-Sachdeva, NeurIPS 2021{]}. Our implementation of
the oracle relies on solving the overdetermined $\ell_{\infty}$-regression
problem $\min_{x\in\R^{d},Nx=v}\|Ax-b\|_{\infty}$. We provide an
algorithm that finds a $(1+\epsilon)$-approximate solution to this
problem using $O((d^{1/3}/\epsilon+1/\epsilon^{2})\log(n/\epsilon))$
linear system solves. This algorithm leverages $\ell_{\infty}$ Lewis
weight overestimates and achieves this iteration complexity via a
simple lightweight IRLS approach, inspired by the work of {[}Ene-Vladu,
ICML 2019{]}. Experimentally, we demonstrate that our algorithm significantly
improves the runtime of the standard CVX solver.
\end{abstract}

\section{Introduction}

Quasi-self-concordant (QSC) optimization encompasses a broad class
of problems that have long been central to convex optimization, numerical
analysis, robust statistics, data fitting, and constrained optimization
\cite{bach2010self,karimireddy2018global,marteau2019globally,sun2019generalized,carmon2020acceleration,carmon2022optimal,doikov2023minimizing}.
Notable special cases include logistic regression, softmax regression,
and regularized $\ell_{p}$ regression, which have been widely studied
due to their broad range of applications in machine learning.

Due to their favorable properties, QSC functions have enjoyed efficient
optimization methods which require strictly weaker curvature assumptions
than, for example, strong convexity. Yet, without strong promises
about the underlying model, QSC optimization is still a challenging
problem for which first order methods only provide low precision solution
with $\mathrm{poly}(1/\epsilon)$ convergence rate as opposed to the
desirable $\mathrm{poly}\log(1/\epsilon)$ rate, where $\epsilon$
is the sub-optimality gap, while typical second order methods are
computationally expensive.  Thus significant challenges remain for
the pursuit of more efficient optimization algorithms for this class
of functions. 

In this work, we study the constrained QSC optimization problem
\begin{align}
\min_{x\in\R^{d},Nx=v} & \ h\left(x\right):=\sum_{i=1}^{n}f\left(\left(Ax-b\right)_{i}\right)\,,\label{eq:problem}
\end{align}
where $f:\R\to\R$ is a QSC function, $A\in\mathbb{R}^{n\times d},N\in\mathbb{R}^{m\times d}$
are matrices, and $b\in\mathbb{R}^{n}$, $v\in\mathbb{R}^{m}$ are
vectors. Each term $f\left(\left(Ax-b\right)_{i}\right)$ can be viewed
as measuring the fit to a single datapoint. We focus in particular
on the overdetermined regime where the number of datapoints $n$ is
significantly larger than the number of variables $d$. This setting
is central to large-scale data analysis. In regression tasks, standard
techniques for managing large $n$ include data sparsification \cite{spielman2008graph,cohen2015uniform},
subspace embeddings \cite{nelson2013osnap}, and tools from randomized
numerical linear algebra \cite{drineas2018lectures}. However, for
QSC objectives, existing approaches typically exhibit iteration complexity
that scales with $n$ rather than the intrinsic dimension $d$, which
makes them poorly suited to modern high-sample regimes. Thus, the
central question to our work is:
\begin{center}
\textit{Can we design an algorithm for Problem (\ref{eq:problem})
with iteration complexity nearly independent of $n$?}
\par\end{center}

We address this question by showing that QSC functions can be minimized
to high precision ($\mathrm{poly}\log(1/\epsilon)$ convergence rate)
in $\widetilde{O}(d^{1/3})$ iterations, each of which makes a constant
number of calls to a linear system solver involving structured $d\times d$
matrices, thereby drastically improving efficiency and scalability
of existing approaches. Our approach is based on a trust-region method,
for which we show a fast oracle implementation. Our oracle implementation
relies on solving $\ell_{\infty}$ regression problems. Instead of
leveraging data sparsification to reduce problem size \cite{cohen2015lp,jambulapati2023sparsifying,jambulapati2024sparsifying,woodruff2024nearly},
which often entails high computational costs when $n$ large, our
algorithm is inspired by the recent advances on general $\ell_{p}$
regression, such as the work of Jambulapati-Liu-Sidford \cite{jambulapati2022improved}.
We introduce a new approach that leverages $\ell_{\infty}$ Lewis
weights inside a simple, lightweight Iteratively Reweighted Least
Squares (IRLS) algorithm \cite{ene2019improved}, which is of independent
interest. A simple modification of our algorithm can be used to solve
QSC optimization problems in the underdetermined regime $n\leq d$
as well, and it provides a much simpler and practical algorithm with
theoretical guarantees that mirror those of the state of the art algorithm
of \cite{adil2021unifying} that currently lacks a practical implementation.

\subsection{Our contributions}

 We consider the QSC optimization problem (\ref{eq:problem}), where
$f$ is a general $C$-QSC function (Definition \ref{def:qsc}), potentially
non-smooth and not strongly convex. Following the literature on regression
algorithms \cite{adil2019fast,adil2019iterative,jambulapati2022improved}
and the prior work by \cite{adil2021unifying}, we focus on designing
highly efficient iterative algorithms that rely on (structured) linear
system solvers. Throughout this paper, we define one iteration as
a single call to the linear system solver. Since the linear system
solves dominate the running time in all algorithms we discuss, the
total runtime is essentially the product of the iteration count and
the cost of solving these systems.

Our first main contribution is a new trust-region method for affine-constrained
$C$-QSC optimization that achieves the following performance guarantee.
\begin{thm}
\label{thm:QSC-main-theorem} Let $f:\R\to\R$ be a $C$-quasi-self-concordant
function, let $A\in\R^{n\times d}$, $N\in\R^{m\times d}$, $b\in\R^{n}$
and $v\in\R^{m}$ with $n>d$. Define the function $h:\R^{d}\to\R$
as $h\left(x\right)=\sum_{i=1}^{n}f\left(\left(Ax-b\right)_{i}\right).$
Let $x^{(0)}$ be an initial solution satisfying $Nx^{(0)}=v$, and
let $R$ be any value such that $\max_{x\in\left\{ x:h\left(x\right)\le h\left(x^{(0)}\right)\right\} }\left\Vert Ax-Ax^{*}\right\Vert _{\infty}\le R$.
Suppose $h$ is bounded from below, and let $B$ be any value satisfying
$h(x)\geq B$ for all $x\in\R^{d}$. Then there exists an algorithm
that, given $x^{(0)}$, $R$, and $B$ as input, it computes an $\epsilon$-additive
approximation to the problem $\min_{x\in\R^{d},Nx=v}h(x)$ by solving
$O\left(CR\log\left(CR\right)\log^{2}\left(\frac{h(x^{(0)})-h(x^{*})}{\epsilon}\right)\right)$
subproblems, each of which makes $O\left(d^{1/3}\log n\right)$ calls
to a linear system solver with matrices of the form $A^{\top}DA$
and $N\left(A^{\top}DA\right)^{+}N^{\top}$, where $D$ is a positive
diagonal matrix. %
\end{thm}
\textbf{Comparison to prior works}: Prior to our work, the algorithm
with state of the art iteration complexity for general (non-smooth)
QSC optimization was given by \cite{adil2021unifying}, which achieved
iteration complexity $O\left(n^{1/3}\log^{O(1)}\left(n\right)\cdot CR\log\left(CR\right)\log\left(\frac{h(x^{(0)})-h(x^{*})}{\epsilon}\right)\right)$.
Our algorithm improves this to a dependence on $d^{1/3}$, which is
a substantial improvement in the overdetermined regime $n\gg d$.
Moreover, the algorithm of \cite{adil2021unifying} is complex and
makes use of several parameters set to guarantee the theoretical runtime.
In practice, to obtain an efficient algorithm, these parameters need
to be carefully tuned and currently this algorithm of \cite{adil2021unifying}
lacks a practical implementation for this reason. On the other hand,
our algorithm is significantly simpler and implementable in practice,
with an excellent empirical performance compared to CVX.

Our algorithm builds on a trust-region based algorithmic framework
developed in prior work \cite{allen2017much,cohen2017matrix,carmon2020acceleration,adil2021unifying}.
An important property of QSC functions shown by \cite{karimireddy2018global}
is that the Hessian at a point $x$ is relatively stable in a region
around $x$. Prior works such as \cite{carmon2020acceleration,adil2021unifying}
build on this property and design algorithms that iteratively minimize
a local second-order approximation of the function in a region where
the Hessian is sufficiently stable. The key difficulty in implementing
this approach lies in designing efficient algorithms for solving the
resulting subproblems. One of our main contributions is a new efficient
subroutine for solving the resulting subproblems. We develop our subroutine
by first designing a novel and simple algorithm for solving $\ell_{\infty}$
regression, which is of independent interest. Our subroutine algorithm
then solves $\ell_{2}$-regularized $\ell_{\infty}$ regression problems,
enjoying the simplicity and efficiency of our $\ell_{\infty}$ regression
solver. The following theorem states the guarantees for our $\ell_{\infty}$
regression algorithm. 
\begin{thm}
\label{thm:Linf-regression}Let $A\in\R^{n\times d}$, $N\in\R^{m\times d}$,
$b\in\R^{n}$ and $v\in\R^{m}$, with $n>d$, and let $\epsilon>0$
be a scalar. There exists an algorithm which provides a $(1+\varepsilon)$-multiplicative
approximation to the $\ell_{\infty}$ regression problem $\min_{x\in\R^{d},Nx=v}\left\Vert Ax-b\right\Vert _{\infty}$
by solving $O\left(\log\frac{\log\left(n\right)}{\epsilon}\right)$
subproblems, each of which makes $O\left(\left(\frac{d^{1/3}}{\epsilon}+\frac{1}{\epsilon^{2}}\right)\log\frac{n}{\epsilon}\right)$
calls to a solver for structured linear systems involving matrices
of the form $A^{\top}DA$ and $N\left(A^{\top}DA\right)^{+}N^{\top}$,
where $D$ is a positive diagonal matrix. 
\end{thm}
Our algorithm for $\ell_{\infty}$ regression achieves state of the
art iteration complexity (in terms of dependence on the dimension
$d$) via an iteratively reweighted least squares (IRLS) approach
that solves a weighted least squares instance in each iteration, whose
solution is provided by a single linear system solve. IRLS algorithms
are favored in practice due to their simplicity and ease of implementation.
IRLS algorithms were first introduced in pioneering work from the
1960s \cite{lawson1961contributions,rice1968lawson}. Subsequently,
there has been significant interest in developing IRLS methods with
provable convergence guarantees \cite{StraszakV16,StraszakV16Soda,ene2019improved,adil2019fast}.

\textbf{Comparison to prior works}: Prior to our work, the state of
the art IRLS algorithm for $\ell_{\infty}$ regression is the algorithm
of \cite{ene2019improved}. The work \cite{jambulapati2022improved}
gives a non-IRLS algorithm for $\ell_{\infty}$ regression with the
state of the art iteration complexity $\widetilde{O}\left(d^{1/3}/\epsilon^{2/3}\right)$.
We highlight the differences between our algorithm and these works.
The algorithm of \cite{jambulapati2022improved} works by leveraging
$\ell_{\infty}$ Lewis weights, and uses a reduction to the Monteiro-Svaiter
accelerated gradient descent algorithm by \cite{carmon2020acceleration}.
Monteiro-Svaiter acceleration is a complex scheme, requiring to solve
an implicit equation in each iteration. While an implementation is
later provided in \cite{carmon2022optimal}, the practicality of Monteiro-Svaiter
acceleration remains an open question \cite{carmon2022optimal}. Our
$\ell_{\infty}$ regression algorithm also uses Lewis weights, but
from a completely different approach, based on an IRLS algorithm.
The starting point of our algorithm is the algorithmic framework by
\cite{ene2019improved}. However the algorithm by \cite{ene2019improved}
is very specific to the uniform initialization of the solution and
only achieves iteration complexity $\widetilde{O}\left(n^{1/3}/\epsilon^{2/3}\right)$.
The non-uniform initialization via $\ell_{\infty}$ Lewis weights
in our algorithm immediately breaks the analysis by \cite{ene2019improved}.
To overcome this barrier requires new insights. Our algorithm also
further improves the iteration complexity of \cite{ene2019improved}'s
algorithm to $\widetilde{O}\left(d^{1/3}/\epsilon+1/\epsilon^{2}\right)$
while retaining the simplicity and practicality of an IRLS method.

\textbf{Underdetermined regime $n\leq d$:} A simple modification
of our algorithm can be used in the regime $n\leq d$ and it obtains
a guarantee which is analogous to Theorem \ref{thm:QSC-main-theorem}
but with an iteration complexity that depends on the smaller dimension
$n$. Our algorithm matches the iteration complexity of the state
of the art algorithm of \cite{adil2021unifying} up to poly-logarithmic
factors, while being significantly simpler and more practical. As
discussed above, our algorithm is an IRLS based method where the main
computation in each iteration is a weighted least squares regression,
which is beneficial in practice.

\textbf{Further extensions:} It is worth noting that while we followed
the formulation from \cite{adil2021unifying}, our solver can handle
more general functions provided that their Hessian is stable within
any $\ell_{\infty}$ box of fixed radius. This allows, for example,
to minimize functions of the form $\sum f_{i}\left(\left(Ax-b\right)_{i}\right)$
in both the constrained and unconstrained settings, where each $f_{i}$
in the decomposition has a bounded third derivative. The same framework
applies to functions without the decomposable structure, provided
that a similar stability property holds.

\subsection{Related work}

The class of quasi-self-concordant (QSC) functions was introduced
in \cite{bach2010self}, extending the definition of standard self-concordant
functions to logistic regression models. As QSC functions encompass
various other important models, such as softmax and $\ell_{p}$ regressions,
they have been further explored in a series of influential works \cite{sun2019generalized,karimireddy2018global,carmon2020acceleration,adil2021unifying,doikov2023minimizing}.

Most of these studies rely on trust-region Newton methods, which inherently
introduce a dependence on the number of iterations dictated by the
shape and size of the trust region. Typically, these regions are $\ell_{2}$
balls centered at the current iterate, leading to convergence guarantees
expressed in terms of the $\ell_{2}$ diameter of the search region.
Notably, \cite{adil2021unifying} applied the techniques of \cite{allen2017much,cohen2017matrix}
to instead establish guarantees in terms of the $\ell_{\infty}$ diameter
of the search region by solving a sequence of second-order subproblems
with $\ell_{\infty}$ constraints. Our algorithm for QSC optimization
follows a similar principle.

A significant contribution by \cite{doikov2023minimizing} introduced
acceleration techniques to achieve an iteration complexity of $\widetilde{O}((CR)^{2/3})$,
where $R$ denotes the $\ell_{2}$ diameter of the domain. However,
since the worst-case ratio between the $\ell_{2}$ and $\ell_{\infty}$
diameters is $\sqrt{n}$, the iteration complexity may degrade by
a factor of $n^{1/3}$ when measured under the $\ell_{\infty}$ benchmark
used in \cite{adil2021unifying} and in this work---while still maintaining
the desirable sublinear dependence on $C$ and $R$. In comparison,
our approach incurs only a $d^{1/3}$ factor, which represents a substantial
improvement when $n\gg d$. Additionally we note that accelerated
algorithms typically require smoothness, and this is the case for
those given in \cite{carmon2020acceleration,doikov2023minimizing}.
However, both the algorithm from \cite{adil2021unifying} and ours
manage to achieve an improved iteration complexity, without resorting
to smoothness or other additional properties. We note that \cite{karimireddy2018global}
provide an analysis for both the smooth and non-smooth settings, but
they exhibit a dependence on the $\ell_{2}$ diameter of the level
set of the initialization point, which scales as $n^{1/2}$ in the
worst case.

$\ell_{\infty}$ regression is a classical optimization problem, but
the first major breakthrough occurred in the special case of max flow
in \cite{ChristianoKMST11}, where the authors developed an algorithm
requiring $\widetilde{O}\left(n^{1/3}\right)$ linear system solves,
improving upon the standard $\widetilde{O}\left(n^{1/2}\right)$ complexity.
This result was further refined and generalized in \cite{ChinMMP13,ene2019improved},
both achieving $\widetilde{O}\left(n^{1/3}\right)$ iteration complexity.
More recently, \cite{jambulapati2022improved} studied both $\ell_{p}$
regression for general $p<\infty$ and $\ell_{\infty}$ regression,
presenting algorithms whose iteration count depends on $d$ rather
than $n$, leveraging Lewis weights. Previously, \cite{lee2014path}
devised an interior point method for tall LPs where the number of
iterations depends on $d$ rather than $n$ by employing Lewis weights.
For the specific case of $\ell_{\infty}$ regression, \cite{jambulapati2022improved}
incorporated Lewis weights into the accelerated second-order method
of \cite{carmon2020acceleration}, which approximates the $\ell_{\infty}$
norm via the Hessian of a log-exp-sum function. In contrast, while
we also draw on the Lewis weight approach from \cite{jambulapati2022improved},
we develop a simpler, lightweight IRLS method inspired by \cite{ene2019improved}.

\subsection{Overview of our approach}

To establish our main result, we employ a trust-region method. While
such methods have been extensively studied, the distinguishing feature
of our approach is that the convergence rate depends on the $\ell_{\infty}$
diameter of the level set that contains the initialization point,
rather than its $\ell_{2}$ diameter, which can be larger by up to
a factor of $\sqrt{n}$. This yields both theoretical and practical
improvements. Our method is motivated by the box-constrained Newton
method which was introduced by \cite{cohen2017matrix,allen2017much},
and later adapted to quasi-self-concordant minimization \cite{adil2021unifying}.
This perspective allows us to reduce the optimization process to executing
an outer loop of $O\left(CR\log\frac{h\left(x_{0}\right)-h\left(x^{*}\right)}{\varepsilon}\right)$
iterations, where $C$ is the QSC parameter of $f$ and $R=\max\{\|Ax-Ax^{*}\|_{\infty}:h(x)\leq h(x_{0})\}$
denotes the $\ell_{\infty}$ diameter of the level set of the initialization
point $x_{0}$.

In each iteration of the outer loop, we compute the update by approximately
minimizing a quadratic approximation of $h$ over the $\ell_{\infty}$
ball of radius $1/C$ centered at the current iterate. This radius
constraint guarantees that the Hessian of $h$ remains well approximated
within the entire trust region. Consequently, the approximate minimizer
computed here achieves sufficient decrease in objective value compared
to the true minimizer of the region.

We then show that the required subproblem is of the form 
\begin{align}
\max_{\|A\Delta\|_{\infty}\le\frac{1}{C}} & g^{\top}\left(A\Delta\right)-\frac{1}{2}\left(A\Delta\right)^{\top}D\left(A\Delta\right)\,,\label{eq:residual-problem-1}
\end{align}
where $g\in\mathbb{R}^{m}$ is a vector and $D$ is a non-negative
diagonal matrix and can be reduced to implementing a residual solver
which approximately maximizes a concave quadratic objective as in
(\ref{eq:residual-problem-1}) subject to $\ell_{\infty}$ constraints.
We describe this reduction in Section \ref{sec:Generalization-to-QSC}.
To obtain an efficient residual solver, we design an IRLS method whose
convergence rate depends only on the smaller dimension $d$, requiring
only $\widetilde{O}\left(d^{1/3}\right)$ linear system solves. Finally,
the extra factor $\log\left(CR\right)\log\left(\frac{h(x^{(0)})-h(x^{*})}{\epsilon}\right)$
comes from a binary search procedure, detailed in Section \ref{sec:Generalization-to-QSC}.

\section{Preliminaries\protect\label{sec:Preliminaries}}

\textbf{Notation}. Throughout the paper, when it is clear from context,
we use scalar operations between vectors to denote the coordinate-wise
application, e.g, for $x\in\R^{d}$, $x^{2}$ is an $\R^{d}$ vector
with entry $x_{i}^{2}$ in the $i$-th coordinate. For $w\in\R^{d},$we
use $\diag\left(w\right)\in\R^{d\times d}$ to denote the diagonal
matrix with diagonal entries given by $w$. When it is clear from
context, we also abuse the notation and use the scalar $a$ to denote
the vector with all coordinates having value $a$, and the dimension
will be inferred from the context.

\textbf{Quasi-self-concordant functions}. We recall the definition
of quasi-self-concordant functions.
\begin{defn}
\label{def:qsc}A function $f:\R\to\R$, is $C$-quasi-self-concordant
for $C>0$ if for all $x\in\R$, $\left|f'''(x)\right|\le Cf''(x)$. 
\end{defn}
\textbf{Lewis weights and computation}. Our main algorithms rely on
the use of $\ell_{\infty}$ Lewis weight overestimates. To start,
we introduce leverage scores and $\ell_{\infty}$ Lewis weights for
a matrix $A\in\R^{n\times d}$.
\begin{defn}[Leverage scores]
 For a matrix $A$, the leverage score of the $i$-th row of $A$
is given by $\sigma\left(A\right){}_{i}=a_{i}^{\top}\left(A^{\top}A\right){}^{-1}a_{i}$
where $a_{i}$ is the $i$-th row of $A$.
\end{defn}
\begin{defn}[$\ell_{\infty}$ Lewis weights]
 The $\ell_{\infty}$ Lewis weights of matrix $A$ are the unique
vector $w\in\R_{\ge0}^{n}$ that satisfies $w_{i}=\sigma\left(\diag(w)^{1/2}A\right){}_{i}$
for all $i\in[n]$.
\end{defn}
While finding $\ell_{\infty}$ Lewis weights is computationally expensive,
for the purpose of our algorithms, we only need to use $\ell_{\infty}$
Lewis weight overestimates, a notion introduced by \cite{jambulapati2022improved}.
\begin{defn}[$\ell_{\infty}$ Lewis weight overestimates]
 The $\ell_{\infty}$ Lewis weight overestimates of matrix $A$ are
a vector $w\in\R_{\ge0}^{n}$ that satisfies $d\le\|w\|_{1}\le2d$
and $w_{i}\ge\sigma\left(\diag(w)^{1/2}A\right){}_{i}$ for all $i\in[n]$.
\end{defn}
One can efficiently compute $\ell_{\infty}$ Lewis weight overestimates
of matrix $A$ by the fixed point iterations by \cite{cohen2019near},
using $\widetilde{O}(1)$ linear system solves (see also \cite{jambulapati2022improved}).

\textbf{Linear system solver}. Our algorithms in the unconstrained
setting assume access to an oracle that solves problems of the form
$\min_{x:g^{\top}x=-1}\langle r,(Ax)^{2}\rangle$. Finding the minimizer
$x$ is equivalent to solving a linear system of the form $A^{\top}DAx=g$,
for a diagonal matrix $D$. The problem has a closed form solution
$x=-\frac{B^{-1}g}{g^{\top}B^{-1}g}$ where $B=A^{\top}\diag(r)A$.

\textbf{Energy increase lower bounds}. Our approach is based on the
electrical flow interpretation of the problem. Electrical flow interpretation
was developed for flow problems on graph, and later extended for problems
involving general linear systems. For vector $g\in\R^{d}$, matrix
$A\in\R^{n\times d}$, and vector $r\in\R_{\ge0}^{n}$ which we will
refer to as resistances, we consider the following quantity $\cE(r)=\min_{x:g^{\top}x=-1}\langle r,(Ax)^{2}\rangle$,
which we will refer to as the electrical energy. $\cE(r)$ has the
nice property of being monotone in $r$. The increase in the energy
when the resistances increase is lower bounded in the two lemmas \ref{lem:energy-usual}
and \ref{lem:energy-lewis-weight} (provided in the appendix), which
are both critical in the design and analysis of our algorithms.

\section{$\ell_{\infty}$-Regression with Lewis Weights\protect\label{sec:Regression-with-Lewis-weight}}

As a warm up for the general QSC optimization algorithm, in this section,
we solve the problem of overdetermined $\ell_{\infty}$-regression
to a $(1+\epsilon)$-approximation, in the form
\begin{align}
\min_{g^{\top}x=-1} & \|Ax\|_{\infty}\label{eq:Linf-problem}
\end{align}
where $A\in\R^{n\times d}$ with $n\ge d$. That is, we want to find
a solution $x$ to the above problem that satisfies $g^{\top}x=-1$
and $\|Ax\|_{\infty}\le(1+\epsilon)\|Ax^{*}\|_{\infty}$, where we
denote $x^{*}=\arg\min_{g^{\top}x=-1}\|Ax\|_{\infty}$. We note that
the commonly seen problem $\min_{x}\left\Vert Ax-b\right\Vert _{\infty}$
can be transformed into the above form by adding $b$ as a column
to $A$ and considering $x\in\R^{d+1}$, with the extra dimension
having value $-1$. 

\subsection{Algorithm}

\begin{algorithm}
\caption{$\ell_{\infty}$-regression for $\min_{g^{\top}x=-1}\|Ax\|_{\infty}$}

\label{alg:Linf-regression-main}

\begin{algorithmic}[1]

\STATE \textbf{Input}:$A$, $g$, $\epsilon$

\STATE \textbf{Output}: Find $x$ such that $g^{\top}x=-1$ and $\|Ax\|_{\infty}\le(1+\epsilon)\min_{g^{\top}x=-1}\|Ax\|_{\infty}$ 

\STATE Initialize $x^{(0)}=\arg\min_{g^{\top}x=-1}\|Ax\|_{2}$; $L=\lfloor\log_{1+\epsilon}\frac{\|Ax^{(0)}\|_{2}}{n^{1/2}}\rfloor$;
$U=\lfloor\log_{1+\epsilon}\|Ax^{(0)}\|_{2}\rfloor$

\STATE \textbf{while} $L<U$:

\STATE $\qquad$$P=\lfloor\frac{L+U}{2}\rfloor$; $M=(1+\epsilon)^{P}$

\STATE $\qquad$\textbf{if} $\mathsf{Subsolver}(A,g,\epsilon,M)$
is infeasible then $L=P+1$

\STATE $\qquad$\textbf{else }let $x^{(t)}$ be the output of $\mathsf{Subsolver}(A,g,\epsilon,M)$;
$U=P$, $t\gets t+1$

\STATE \textbf{return} $x^{(t)}$

\end{algorithmic}
\end{algorithm}
\begin{algorithm}[th]
\caption{$\ell_{\infty}$-regression $\mathsf{Subsolver}(A,g,\epsilon,M)$}

\label{alg:Linf-regression}

\begin{algorithmic}[1]

\STATE\textbf{Output}: Find $x$ such that $\|Ax\|_{\infty}\le(1+\epsilon)M$
or find $r$ such that $\frac{\cE(r)}{\|r\|_{1}}\ge\big(\frac{M}{1+\epsilon}\big)^{2}$

\STATE\textbf{Initialize}: $r^{(0)}=w+\frac{d}{n}$, where $w$ is
an $\ell_{\infty}$ Lewis weight overestimate vector of $A$

\STATE$t=0$, $t'=0$, $s^{(t')}=0$

\STATE\textbf{while} $\|r^{(t)}\|_{1}\leq\frac{\|r^{(0)}\|_{1}}{\epsilon}$

\STATE$\qquad x^{(t)}=\arg\min_{x:g^{\top}x=-1}\langle r^{(t)},(Ax)^{2}\rangle$

\STATE$\qquad$\textbf{if} $\frac{\langle r^{(t)},(Ax^{(t)})^{2}\rangle}{\|r^{(t)}\|_{1}}\ge\big(\frac{M}{1+\epsilon}\big)^{2}$
\textbf{then} \textbf{return} $r^{(t)}$

\STATE$\qquad$\textbf{if} $\|Ax^{(t)}\|_{\infty}\le(1+\epsilon)M$
\textbf{then} \textbf{return} $x^{(t)}$

\STATE$\qquad$\textbf{if} $\|Ax^{(t)}\|_{\infty}>S=d^{\frac{1}{3}}M$:\textbf{\hfill{}}$\triangleright$
\emph{Case 1}

\STATE$\qquad\qquad$Let $i$ be an index such that $|(Ax^{(t)})_{i}|=\|Ax^{(t)}\|_{\infty}$

\STATE$\qquad\qquad$$r_{j}^{(t+1)}=\begin{cases}
r_{j}^{(t)} & j\neq i\\
r_{j}^{(t)}+1 & j=i
\end{cases}$

\STATE$\qquad$\textbf{else}:\textbf{\hfill{}}$\triangleright$
\emph{Case 2}

\STATE$\qquad\qquad$Let $t'=t'+1$; $s^{(t')}=s^{(t'-1)}+x^{(t)}$;

\STATE$\qquad\qquad$\textbf{if} $\|As^{(t')}\|_{\infty}/t'\le(1+\epsilon)M$
\textbf{then} \textbf{return} $s^{(t')}/t'$

\STATE$\qquad\qquad$$r_{j}^{(t+1)}=\begin{cases}
r_{j}^{(t)}\frac{(Ax^{(t)})_{j}^{2}}{M^{2}} & \text{if }(Ax^{(t)})_{j}^{2}\ge(1+\epsilon)M^{2}\\
r_{j}^{(t)} & \mbox{otherwise}
\end{cases}$

\STATE$\qquad$$t=t+1$

\STATE\textbf{return} $r^{(t)}$

\end{algorithmic}
\end{algorithm}

The main part of our algorithm is the subroutine shown in Algorithm
\ref{alg:Linf-regression} which takes as input a guess for the optimal
objective of Problem \ref{eq:Linf-problem}. To obtain a $(1+\epsilon)$-approximation
for this value, our main algorithm \ref{alg:Linf-regression-main}
proceeds by performing a binary search. Starting with an initial solution
$x^{(0)}=\arg\min_{g^{\top}x=-1}\|Ax\|_{2}$, the algorithm performs
a binary search on the $(1+\epsilon)$-grid between $\|Ax^{(0)}\|_{2}$
and $\frac{\|Ax^{(0)}\|_{2}}{n^{1/2}}$ (which are an upper bound
and a lower bound for $\|Ax^{*}\|_{\infty}$). The number of guesses
is at most $\log\frac{\log(\|Ax^{(0)}\|_{\infty})-\log(\frac{\|Ax^{(0)}\|_{2}}{n^{1/2}})}{\epsilon}=O(\log\frac{\log n}{\epsilon})$.

Our main focus in this section is then to demonstrate that Algorithm
\ref{alg:Linf-regression}, given a guess $M,$ outputs a solution
$x$ with $\|Ax\|_{\infty}\le(1+\epsilon)M$, if any. For our convenience,
we start with the dual formulation of the problem with the squared
objective
\begin{align*}
\min_{g^{\top}x=-1}\|Ax\|_{\infty}^{2} & =\min_{g^{\top}x=-1}\max_{\|r\|_{1}=1}\left\langle r,(Ax)^{2}\right\rangle =\max_{\|r\|_{1}=1}\min_{g^{\top}x=-1}\left\langle r,(Ax)^{2}\right\rangle =\max_{r\ge0}\frac{\cE(r)}{\|r\|_{1}},
\end{align*}
where we use $\cE(r)$ to denote the objective $\min_{g^{\top}x=-1}\left\langle r,(Ax)^{2}\right\rangle $.
For a guess $M$ for the optimal objective of Problem \ref{eq:Linf-problem},
the goal of our algorithm is to produce a primal solution $x$ such
that $g^{\top}x=-1$ and $x$ satisfies $\|Ax\|_{\infty}\le(1+\epsilon)M$
or certify that $\min_{g^{\top}x=-1}\|Ax\|_{\infty}^{2}=\max_{r\ge0}\frac{\cE(r)}{\|r\|_{1}}\ge\big(\frac{M}{1+\epsilon}\big)^{2}$,
in which case we need to increase $M$. 

The core idea of our algorithm is to maintain in each iteration $t$
the following invariant 
\begin{align}
\cE(r^{(t+1)})-\cE(r^{(t)}) & \ge M^{2}(\|r^{(t+1)}\|_{1}-\|r^{(t)}\|_{1}).\label{eq:invariant}
\end{align}
The telescoping property of this invariant guarantees that if the
algorithm outputs a dual solution $r^{(T)}$ with $\|r^{(T)}\|_{1}$
significantly large compared to the initial $\|r^{(0)}\|_{1}$, we
will have $\frac{\cE(r^{(T)})}{\|r^{(T)}\|_{1}}\ge\frac{M^{2}}{(1+\epsilon)^{2}}$. 

We leverage the energy increase lower bound lemmas \ref{lem:energy-lewis-weight}
and \ref{lem:energy-usual} to determine the update for the dual solution
$r$. First, the initial dual solution $r^{(0)}$ is set to $w+\frac{d}{n}$,
where $w$ is a $\ell_{\infty}$ Lewis weight overestimate vector
for matrix $A$, which can be efficiently obtained using the fixed
point iteration from \cite{cohen2019near}, which we revise in Section
\ref{sec:approx-lewis} of the appendix. The $\frac{d}{n}$ component
is added to guarantee that the coordinates of the dual solution are
not too small. With this initialization, subsequent update $r^{(t)}$
always satisfies the condition of Lemma \ref{lem:energy-lewis-weight}
and thus we can apply it when necessary. To update the dual solution,
the novelty in our method is to distinguish between the two regimes.
In the high width regime (Case 1), ie, when the corresponding minimizer
$x^{(t)}$ to $\min_{g^{\top}x=-1}\left\langle r,(Ax)^{2}\right\rangle $
satisfies $\|Ax^{(t)}\|_{\infty}>S$, where $S$ is set to $d^{\frac{1}{3}}M$,
we update a single coordinate $i$ that achieves the maximum value
$|(Ax^{(t)})_{i}|=\|Ax^{(t)}\|_{\infty}$ by setting $r_{i}^{(t+1)}=r_{i}^{(t)}+1$.
Lemma \ref{lem:energy-lewis-weight} will guarantee that the invariant
\ref{eq:invariant} holds, and at the same time, $\cE\left(r^{(t)}\right)$
increases fast. In the low width regime (Case 2), when $\|Ax^{(t)}\|_{\infty}\le S$,
we use the lower bound given by Lemma \ref{lem:energy-usual}. In
order to maintain the invariant \ref{lem:energy-lewis-weight}, we
can guarantee for each coordinate $j$, $\frac{(Ax^{(t)})_{j}^{2}r_{j}^{(t)}\big(1-\frac{r_{j}^{(t)}}{r_{j}^{(t+1)}}\big)}{r_{j}^{(t+1)}-r_{j}^{(t)}}\ge M^{2}$.
From here, we can derive the update rule as shown in Algorithm \ref{alg:Linf-regression}.

We give the full description in Algorithm \ref{alg:Linf-regression}.

\subsection{Analysis}

We show the full analysis of Algorithm \ref{alg:Linf-regression}
in Appendix \ref{sec:analysis-regression}. To show Theorem \ref{thm:Linf-regression},
first, we can see that the algorithm performs $O\big(\log\frac{\log(\left\Vert Ax_{0}\right\Vert _{\infty}-\left\Vert Ax^{*}\right\Vert _{\infty})}{\epsilon}\big)$
binary search steps. Our analysis shows that, in each step, the algorithm
uses $O\big(\big(\frac{1}{\epsilon^{2}}+\frac{d^{1/3}}{\epsilon}\big)\log\frac{n}{\epsilon}\big)$
calls to the solver for problems of the form $\min_{x:g^{\top}x=-1}\langle r,(Ax)^{2}\rangle$.
By combining these two facts, we immediately have Theorem \ref{thm:Linf-regression}.

\section{Quasi-Self-Concordant Optimization\protect\label{sec:Generalization-to-QSC}}

\begin{algorithm}[th]
\caption{Algorithm for optimizing $h(x)=\sum_{i=1}^{n}f((Ax-b)_{i})$}

\label{alg:QSC-main}

\begin{algorithmic}[1]

\STATE \textbf{Input}: initial solution $x_{0}$, lower bound $B$
on $h$, diameter $R$ of sub-level set ${\cal L}_{0}$.

\STATE \textbf{Output}: Find $x$ such that $h(x)\le h(x^{*})+\epsilon$

\STATE \textbf{Initialize}: $x^{(0)}=x_{0}$

\STATE Let $T=O\left(CR\log\left(\frac{h(x^{(0)})-B}{\epsilon}\right)\right)$

\STATE \textbf{for} $\left(t=0;t<T;t\gets t+1\right)$:\textbf{\label{forloop-t}\hfill{}}$\triangleright$
or terminate when $h(x^{(t+1)})-h(x^{(t)})\le O(\epsilon/CR)$

\STATE $\qquad$$\mathbf{\Delta}\gets\left\{ \emptyset\right\} $\textbf{\hfill{}}$\triangleright$
set of candidate updates

\STATE $\qquad$\textbf{for} $\left(\nu=h(x^{(t)})-B;\text{\ensuremath{\nu}\ensuremath{\geq}\ensuremath{\epsilon}};\nu\gets\frac{1}{2}\nu\right):$\textbf{\label{forloop-nu}\hfill{}}$\triangleright$
halving after each step

\STATE $\qquad\qquad$\textbf{for $\left(M=e^{2}\nu;M\geq\frac{\nu}{CR};M\gets\frac{1}{2}M\right):$\label{forloop-M}\hfill{}}$\triangleright$
halving after each step

\STATE $\qquad\qquad\qquad$Let $\Delta_{\nu,M}$ be a primal solution
output by $\ressolver(x^{(t)},M)$ if any

\STATE $\qquad\qquad\qquad$$\mathbf{\Delta}\gets\mathbf{\Delta}\cup\left\{ \Delta_{\nu,M}\right\} $

\STATE $\qquad\qquad$$x^{(t+1)}=x^{(t)}-\frac{1}{e^{2}}\Delta^{(t)}$
where $\Delta^{(t)}=\arg\min_{\Delta\in\mathbf{\Delta}}h\left(x^{(t)}-\frac{1}{e^{2}}\Delta\right)$\textbf{\hfill{}}$\triangleright$
update step

\STATE 

\end{algorithmic}
\end{algorithm}

\begin{algorithm}[th]
\caption{$\protect\ressolver(x,M)$}

\label{alg:residual-solver}

\begin{algorithmic}[1]

\STATE\textbf{Initialize}: $r^{(0)}=w+\frac{d}{n}$, where $w$ is
an $\ell_{\infty}$ Lewis weight overestimate vector of $A$

\STATE$g_{i}=\frac{-1}{M}(A^{\top}\nabla f(x))_{i}$, $s_{i}=f''((Ax)_{i})$,
$t=0$, $t'=0$, $v^{(t')}=0$

\STATE\textbf{while} $\|r^{(t)}\|_{1}\leq2(\|w\|_{1}+d)$

\STATE$\qquad$Let $p^{(t)}=2(\|w\|_{1}+d)s+\frac{MC^{2}}{2}r^{(t)}$;
$\Delta^{(t)}=\arg\min_{\Delta:g^{\top}\Delta=-1}\left\langle p^{(t)},(A\Delta)^{2}\right\rangle $

\STATE$\qquad$\textbf{if} $\left\langle s+\frac{MC^{2}}{2}\frac{r^{(t)}}{\left\Vert r^{(t)}\right\Vert _{1}},(A\Delta^{(t)})^{2}\right\rangle \ge13M$
\textbf{then} \textbf{return} $r^{(t)}$

\STATE$\qquad$\textbf{if} $\|A\Delta^{(t)}\|_{\infty}\le\frac{11}{C}$
\textbf{then} \textbf{return} $\Delta^{(t)}$

\STATE$\qquad$\textbf{if} $\|A\Delta^{(t)}\|_{\infty}>S=\frac{11d^{\frac{1}{3}}}{C}$:\textbf{\hfill{}}$\triangleright$
\emph{Case 1}

\STATE$\qquad\qquad$Let $i$ be an index such that $|(A\Delta^{(t)})_{i}|=\|A\Delta^{(t)}\|_{\infty}$

\STATE$\qquad\qquad$$r_{j}^{(t+1)}=\begin{cases}
r_{j}^{(t)} & j\neq i\\
r_{j}^{(t)}+1 & j=i
\end{cases}$

\STATE$\qquad$\textbf{else}:\textbf{\hfill{}}$\triangleright$
\emph{Case 2}

\STATE$\qquad\qquad$Let $t'=t'+1$; $v^{(t')}=v^{(t'-1)}+\Delta^{(t)}$

\STATE$\qquad\qquad$\textbf{if} $\|Av^{(t')}\|_{\infty}/t'\le\frac{11}{C}$
\textbf{then} \textbf{return} $v^{(t')}/t'$

\STATE$\qquad\qquad$$r_{j}^{(t+1)}=\begin{cases}
\frac{1}{52}r_{j}^{(t)}(A\Delta^{(t)})_{j}^{2}C^{2} & \text{if }(A\Delta^{(t)})_{j}^{2}\ge\frac{100}{C^{2}}\\
r_{j}^{(t)} & \mbox{otherwise}
\end{cases}$

\STATE$\qquad$$t=t+1$

\STATE\textbf{return} $r^{(t)}$

\end{algorithmic}
\end{algorithm}

To simplify exposition, we present an algorithm for solving the \textit{unconstrained}
problem
\begin{align}
\min_{x} & \ h(x)\coloneqq\sum_{i=1}^{n}f\left((Ax-b)_{i}\right)\label{eq:QSC-problem}
\end{align}
where $f:\R\to\R$ is a $C$-quasi-self-concordant function defined
in Definition \ref{def:qsc}, $A\in\R^{n\times d}$, and $b\in\R^{n}$
with $n\ge d$. In the appendix, we give an extension to the fully
generalized constrained problem $\min_{x:Nx=v}\sum_{i=1}^{n}f\left((Ax-b)_{i}\right)$
as stated in Theorem \ref{thm:QSC-main-theorem}.

\textbf{Notation:} We denote by $x^{*}$ the optimal solution to the
problem and $x^{(0)}$ the initial solution used by our algorithm.
Also, let us write $\nabla f(x)=\left(f'\left((Ax-b)_{1}\right),\dots,f'\left((Ax-b)_{n}\right)\right){}^{\top}$,
$\nabla^{2}f(x)=\diag\left(f''\left((Ax-b)_{1}\right),\dots,f''\left((Ax-b)_{n}\right)\right)$.
We then have $\nabla h(x)=A^{\top}\nabla f(x)$ and $\nabla^{2}h(x)=A^{\top}\nabla^{2}f(x)A$.

\textbf{Assumptions:} Our algorithm and its analysis require the following
assumptions. 
\begin{assumption}
\label{assu:a1}We assume that there exists a finite value $B$ such
that $h(x)\ge B$, for all $x\in\R^{d}$.
\end{assumption}
\begin{assumption}
\label{assu:a2}Let ${\cal L}_{0}=\{x:h(x)\le h(x^{(0)})\}$. We assume
that there exists a finite value $R$ such that $\max_{x\in{\cal L}_{0}}\|Ax-Ax^{*}\|_{\infty}\le R.$
\end{assumption}
We note that Assumption \ref{assu:a1} is standard in prior work (eg.
by \cite{adil2021unifying}). Many commonly used functions are non-negative
and thus admit a trivial lower bound $B=0$. Assumption \ref{assu:a2}
on the boundedness of the diameter of the level set is also common
in prior work on trust-region Newton method, such as \cite{karimireddy2018global,doikov2023minimizing}.
Such dependencies are natural, and virtually all known algorithms
for QSC optimization have a dependency on the size and shape of the
trust region, as well as on the distance between the initial and the
optimal point (e.g. \cite{adil2021unifying} define their distance
parameter such that $\left\Vert Ax^{*}\right\Vert _{\infty}\leq R$
and $\left\Vert Ax^{*}\right\Vert _{2}\leq R$).

\textbf{Our approach:} We follow a trust-region based template. In
particular, our approach leverages the box-constrained Newton method
\cite{cohen2017matrix}, which was used in the context of matrix scaling,
a special case of QSC minimization. A similar method was developed
in parallel by \cite{allen2017much} and was employed by \cite{adil2021unifying}
for optimizing problem \ref{eq:QSC-problem}. The general idea behind
these trust region templates is that, if the Hessian of the objective
is promised to not change by more than a constant multiplicative factor
within an $\ell_{\infty}$ region centered around the current iterate,
then minimizing the local quadratic approximation within this region
yields a new iterate which makes significant progress in function
value. In fact, this progress is comparable to the progress made by
moving to the best point within the $\ell_{\infty}$ region. Hence
the remaining challenge is to approximately solve the quadratic minimization
problem subject to an $\ell_{\infty}$ constraint. Since the subproblem
is very robust to approximation, it suffices to obtain a constant
factor approximation to a regularized $\ell_{\infty}^{2}+\ell_{2}^{2}$
problem, which we achieve via a new IRLS solver.  Now we can provide
formal statements.

Quasi-self-concordance implies Hessian stability, which allows us
to reduce Problem \ref{eq:QSC-problem} to solving a sequence of residual
problems of the form
\begin{align}
\max_{\|A\Delta\|_{\infty}\le\frac{1}{C}}\res_{x}(\Delta) & \coloneqq\nabla f(x)^{\top}(A\Delta)-\frac{1}{e}(A\Delta)^{\top}\nabla^{2}f(x)(A\Delta)\label{eq:residual-problem}
\end{align}
We prove that an approximate solution to Problem \ref{eq:residual-problem}
allows us to make significant progress in the function value of $h$.
The full proof can be found in Section \ref{sec:analysis-QSC-alg}.
\begin{lem}
\label{lem:progress} Consider an iteration $t$ of Algorithm \ref{alg:QSC-main},
and let $x=x^{(t)}$ be the current iterate. Suppose the $\ressolver$
can compute $\widetilde{\Delta}$ such that $\res_{x}(\widetilde{\Delta})\ge\kappa\max_{\|A\Delta\|_{\infty}\le\frac{1}{C}}\res_{x}(\Delta)$.
Then we have
\begin{align*}
h\Big(x-\frac{\widetilde{\Delta}}{e^{2}}\Big)-h(x^{*}) & \le\Big(1-\frac{\kappa}{e^{2}CR}\Big)\left(h(x)-h(x^{*})\right).
\end{align*}
Consequently, Algorithm \ref{alg:QSC-main} constructs a solution
$x^{(T)}$ such that $h(x^{(T)})\le h(x^{*})+\epsilon$ using $T=O\big(\frac{RC}{\kappa}\log\frac{h(x^{(0)})-h(x^{*})}{\epsilon}\big)$
iterations of the outermost for loop.
\end{lem}
To solve Problem \ref{eq:residual-problem}, we recast it in a more
amenable form, which is related to an $\ell_{\infty}$-regression
problem, but with the key difference that it contains an additional
quadratic term $\left\langle s,(A\Delta)^{2}\right\rangle $. Specifically,
given a guess $M$ for the objective value, we define the minimization
problem
\begin{align}
\min_{g^{\top}\Delta=-1} & \left\langle s,(A\Delta)^{2}\right\rangle +\frac{MC^{2}}{2}\|A\Delta\|_{\infty}^{2}\,,\label{eq:new-residual-problem}
\end{align}
where $g_{i}=\frac{-1}{M}\left(A^{\top}\nabla f(x)\right)_{i}$ and
$s_{i}=f''\left((Ax)_{i}\right)$ for all $i$. We show that algorithm
$\ressolver$ satisfies the following invariant.

\begin{invariant}\label{inv:invariant1}The algorithm either outputs
a solution $\Delta$ such that $g^{\top}\Delta=-1$, $\|A\Delta\|_{\infty}\le\frac{11}{C},$
and $\left\langle s,(A\Delta)^{2}\right\rangle \le\min_{g^{\top}\Delta=-1}\left\langle s,(A\Delta)^{2}\right\rangle +\frac{MC^{2}}{2}\|A\Delta\|_{\infty}^{2},$
or certifies that $\min_{g^{\top}\Delta=-1}\left\langle s,(A\Delta)^{2}\right\rangle +\frac{MC^{2}}{2}\|A\Delta\|_{\infty}^{2}\ge13M$.

\end{invariant}

The output of an algorithm satisfying Invariant \ref{inv:invariant1}
can be converted into a good approximate solution to Problem \ref{eq:residual-problem}.
We prove this formally in Lemma \ref{lem:solution-guarantee}. At
this point the major challenge is providing such an algorithm. We
show that we can achieve it by extending our approach for solving
$\ell_{\infty}$-regression from Section \ref{sec:Regression-with-Lewis-weight}.
The key difference in this setting is the presence of an additional
quadratic term, which makes the extension non-trivial and limits us
to achieving only a constant-factor approximation. Nonetheless, this
level of approximation suffices for our purposes.

The following lemma provides formal guarantees for the $\ressolver$,
and its full proof can be found in Section \ref{sec:analysis-QSC-alg}.
\begin{lem}
\label{lem:main-residual-solver-lemma}For an iterate $x,$ let $\Delta^{*}=\arg\max_{\|A\Delta\|_{\infty}\le1/C}\res_{x}(\Delta)$
and $\opt=\res_{x}(\Delta^{*})$. For $M$ such that $\opt\in(\frac{M}{2},M]$,
Algorithm \ref{alg:residual-solver} outputs $\widehat{\Delta}$ such
that $\|A\widehat{\Delta}\|_{\infty}\le\frac{1}{C}$ and $\res_{x}(\widehat{\Delta})\ge\frac{\opt}{20}$,
and makes $O(d^{1/3}\log n)$ calls to a linear system solver.
\end{lem}
Having provided the description of the $\ressolver$ and its guarantee,
we can now describe and analyze the main routine shown in Algorithm
\ref{alg:QSC-main}. The algorithm starts with a solution $x^{(0)}$
and iteratively updates it. In each iteration, it performs a binary
search for the objective of residual problem \ref{eq:residual-problem}
using the two for-loops in Line \ref{forloop-nu} and Line \ref{forloop-M}.
With each guess for this objective, the algorithm uses the $\ressolver$
(Algorithm \ref{alg:residual-solver}) to find a solution. Algorithm
\ref{alg:QSC-main} finally uses the best solution among the ones
that are found to update the iterate. To analyze the algorithm, we
first use the following fact from \cite{adil2021unifying}.
\begin{lem}[Lemma 4.3, 4.4 \cite{adil2021unifying}]
\label{lem:bin-search-opt}If $h(x^{(t)})-h(x^{*})\in(\frac{\nu}{2},\nu]$,
then $\opt\in(\frac{\nu}{8CR},e^{2}\nu]$. Further, if $\opt\in(\frac{M}{2},M]$
then $(A\Delta^{*})^{\top}\nabla^{2}f(x)(A\Delta^{*})\le eM$.
\end{lem}
This lemma guarantees that the steps executed in the two for-loops
in Algorithm \ref{alg:QSC-main} are sufficient to find a sufficiently
close guess for the residual problem objective. In total, the number
of steps required to find $M$ such that $\opt\in(\frac{M}{2},M]$
is $O\left(\log\left(CR\right)\log\left(\frac{h(x^{(t)})-B}{\epsilon}\right)\right).$

Finally, combining Lemma \ref{lem:progress}, Lemma \ref{lem:main-residual-solver-lemma}
with the binary search procedure in Line \ref{forloop-nu}-\ref{forloop-M}
(Lemma \ref{lem:bin-search-opt}), we can conclude that the total
number iterations of Algorithm \ref{alg:QSC-main} to find an $\epsilon$-additive
solution is $O\big(CR\log\left(CR\right)\log^{2}\left(\frac{h(x^{(0)})-h(x^{*})}{\epsilon}\right)\big)$,
each of which uses $O\left(d^{1/3}\log n\right)$ linear system solves.
This concludes the proof for Theorem \ref{thm:QSC-main-theorem}.

\section{Experiments}

\begin{figure*}

\subfloat{\includegraphics[width=0.25\textwidth]{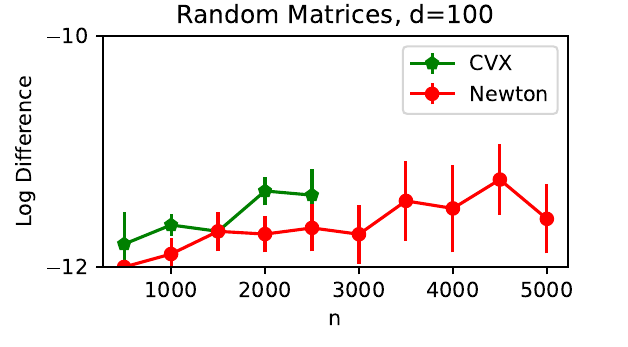}}\subfloat{\includegraphics[width=0.25\textwidth]{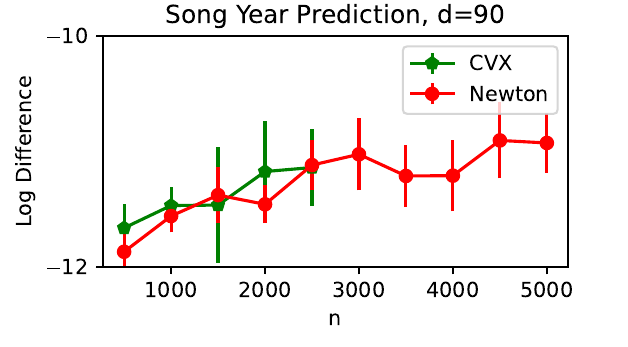}}\subfloat{\includegraphics[width=0.25\textwidth]{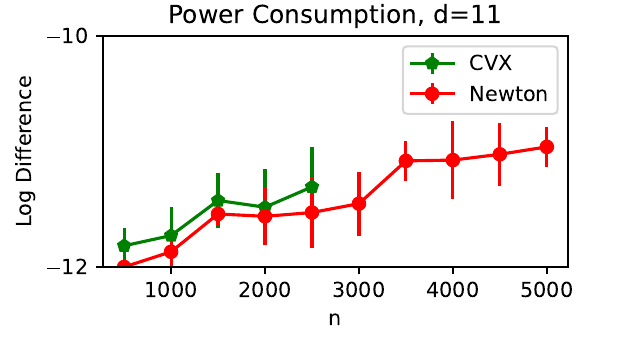}}\subfloat{\includegraphics[width=0.25\textwidth]{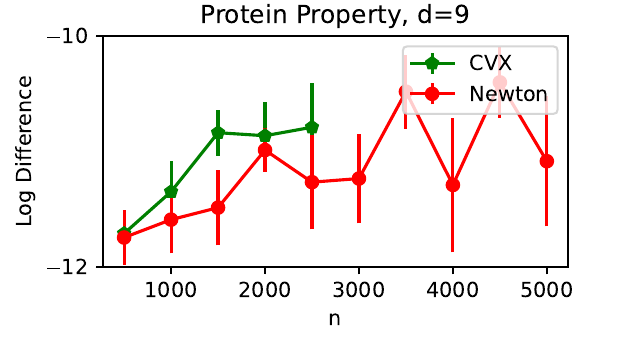}}

\subfloat{\includegraphics[width=0.25\textwidth]{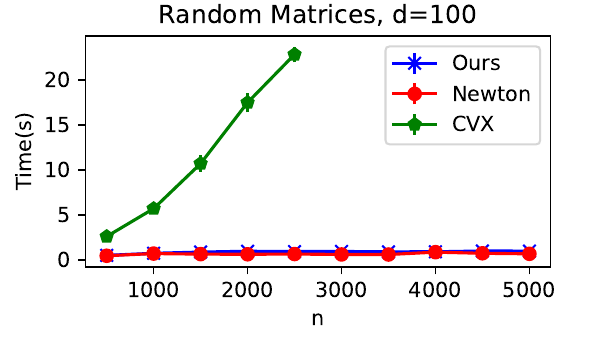}}\subfloat{\includegraphics[width=0.25\textwidth]{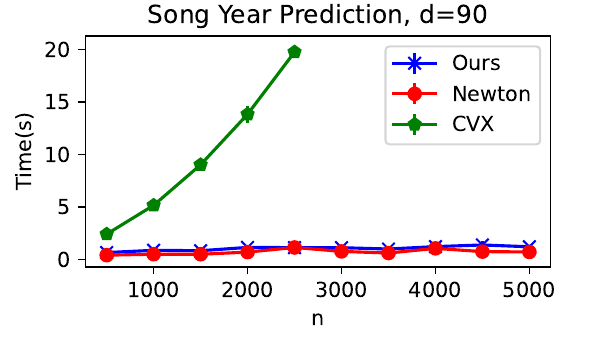}}\subfloat{\includegraphics[width=0.25\textwidth]{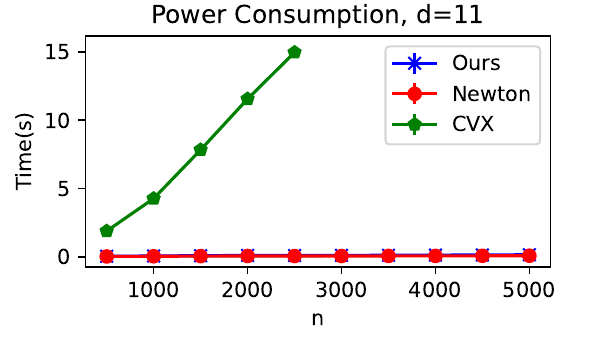}}\subfloat{\includegraphics[width=0.25\textwidth]{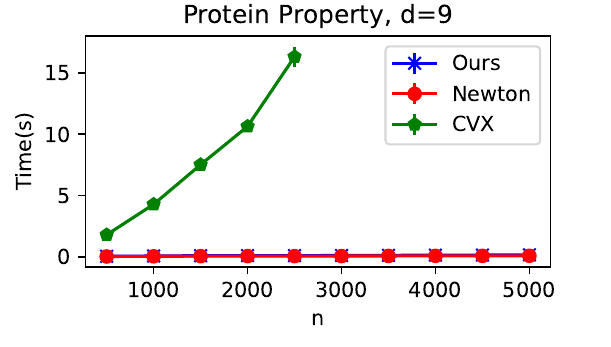}}

\caption{Runtime (in seconds) on random matrices and three real-world datasets
when $p=8$ and $\epsilon=10^{-10}$. We run each each experiment
5 times and report the mean and standard deviation of the runtime.
The first row shows the absolute difference in the objectives returned
by CVX and Newton's method against our algorithm on the $\log$ scale.}
\label{plot:plot}
\end{figure*}

\begin{table}
\caption{Performance of our algorithm on the large instances. In the last row,
we show the gap in the objective value between Newton's method and
our algorithm. }
\label{tab:real-world-info}

\begin{tabular}{|c|>{\centering}p{0.17\columnwidth}|>{\centering}p{0.19\columnwidth}|>{\centering}p{0.18\columnwidth}|>{\centering}p{0.18\columnwidth}|}
\hline 
Dataset & Random Matrices & Song Year Prediction \cite{year_prediction_msd_203} & Consumption of Power \cite{individual_household_electric_power_consumption_235} & Protein Property

\cite{physicochemical_properties_of_protein_tertiary_structure_265}\tabularnewline
\hline 
\hline 
Size & $100000\times100$ & $463811\times90$ & $1844352\times11$ & $41157\times9$\tabularnewline
\hline 
Newton (s) & 0.5 & 2.7 & 3.0 & 0.2\tabularnewline
\hline 
Ours (s) & 1.3 & 7.5 & 6.5 & 0.3\tabularnewline
\hline 
Gap & $-3.6\times10^{-12}$ & $2.6\times10^{-9}$ & $4.5\times10^{-9}$ & $-7.1\times10^{-11}$\tabularnewline
\hline 
\end{tabular}
\end{table}

We test our algorithm on $\ell_{2}$-regularized $\ell_{p}$-regression
problems: $\min_{x}\left\Vert Ax-b\right\Vert _{p}^{p}+\mu\left\Vert Ax-b\right\Vert _{2}^{2}$.
For $p\ge3$, the function $\left|x\right|^{p}+\mu x^{2}$ is $C$-QSC
for $C=p\mu^{-1/(p-2)}$. 

\textbf{Baseline}. We compare our algorithm runtime against CVX and
the Newton's method in \cite{karimireddy2018global}. For the Newton's
method, we use a line search to set the step size in each iteration.

\textbf{Problems} \textbf{and datasets}. We study the $\ell_{p}+\ell_{2}$
regression problem in two settings: 1) Random matrix $A,b$: The entries
of $A,b$ are generated uniformly at random between $0$ and $1$,
the second dimension ($d$) of $A$ is fixed to $100$; 2) On real
world datasets: we use for real-world datasets available for regression
tasks on UCI repository. The dataset sizes are reported in Table \ref{tab:real-world-info}.
Since CVX slows down significantly on larger instances, we only randomly
pick up to 2500 in each instance. For our algorithm and the Newton's
method, we also measure the time they run on larger instances and
the entire datasets. In all experiments, we use $p=8$ and precision
$\epsilon=10^{-10}$, $\mu=1$. 

\textbf{Correctness}. To verify the correctness of our algorithms,
we assume that CVX gives a solution with high precision. We plot the
absolute difference between the objectives outputted by CVX and Newton's
method and the objective by our algorithm. In all cases, the difference
is within the margin of error ($\epsilon=10^{-10}$), indicating that
our algorithm outputs high-precision solutions.

\textbf{Performance}. Implementations were done on MATLAB 2024a on
a MacBook Pro M2/16GB RAM. The performance of all three algorithms
is reported in Figure \ref{plot:plot}. On all small instances, our
algorithm and the Newton's method have comparable runtime and are
both significantly faster than CVX. 

We further compare the performance of our algorithm and the Newton's
on large instances (reported in Table \ref{tab:real-world-info}).
Both algorithms are efficient, and are even faster than the time CVX
runs on small instances. However, there is a gap in the runtime between
our algorithm and Newton's method. We note that Newton's method has
been shown to perform very well in practice. Nevertheless, it comes
with a significantly weaker theoretical guarantee \cite{boyd2004convex}.
Our algorithm, on the other hand, provides a stronger theoretical
convergence rate and, as a proof of concept, gives a close comparable
performance.

\bibliographystyle{alpha}
\bibliography{ref}

\appendix

\section{Energy Lemmas}
\begin{lem}
\label{lem:energy-usual} Given $r,r'\in\R^{n}$ and $r'\ge r$. Let
$x=\arg\min_{x:g^{\top}x=-1}\langle r,(Ax)^{2}\rangle$, then
\begin{align*}
\cE(r')-\cE(r) & \ge\sum_{i=1}^{n}(Ax)_{i}^{2}r_{i}\big(1-\frac{r_{i}}{r_{i}'}\big).
\end{align*}
\end{lem}
The proof for this lemma can be found in prior works such as \cite{ene2019improved}
(for a special case) and \cite{adil2021unifying} (Lemma C.1). 

Via $\ell_{\infty}$ Lewis weight overestimates, \cite{jambulapati2022improved}
also show the following
\begin{lem}[Lemma 3.6 \cite{jambulapati2022improved}]
\label{lem:energy-lewis-weight}Let $w$ be an $\ell_{\infty}$ Lewis
weight overestimates vector for matrix $A$. For $r\ge w$, let $x=\arg\min_{x:g^{\top}x=-1}\langle r^{(t)},(Ax)^{2}\rangle$.
Then for $v\in\R_{\ge0}^{n}$ such that $\|v\|_{1}\le1$
\begin{align*}
\cE(r+v)-\cE(r) & \ge\frac{1}{2}\sum v_{i}(Ax)_{i}^{2}.
\end{align*}
\end{lem}

\section{Analysis of the $\ell_{\infty}$-Regression Algorithm\protect\label{sec:analysis-regression}}

In this section, we give the analysis of the main subroutine of our
$\ell_{\infty}$-regression algorithm, shown in Algorithm \ref{alg:Linf-regression}
.

\subsection{Correctness}

Now we will show the correctness of our algorithm. That is, Algorithm
\ref{alg:Linf-regression} produces a primal solution $x$ such that
$g^{\top}x=-1$ and $x$ satisfies $\|Ax\|_{\infty}\le(1+\epsilon)M$
or a dual solution $r$ with $\frac{\cE(r)}{\|r\|_{1}}\ge\big(\frac{M}{1+\epsilon}\big)^{2}$.

First, we can see that if Algorithm \ref{alg:Linf-regression} outputs
a primal solution $x$ in Line 10 and Line 17, we immediately have
$\|Ax\|_{\infty}\le(1+\epsilon)M$. Therefore, we only have to show
guarantees for the dual solution. As hinted, we first show that Algorithm
\ref{alg:Linf-regression} maintains invariant \ref{eq:invariant}
in the following lemma.
\begin{lem}
\label{lem:Linf-invariant}Algorithm \ref{alg:Linf-regression} maintains
that for all $t\ge0$,
\begin{align*}
\frac{\cE(r^{(t+1)})-\cE(r^{(t)})}{\|r^{(t+1)}\|_{1}-\|r^{(t)}\|_{1}} & \ge M^{2}.
\end{align*}
\end{lem}
\begin{proof}
When $\|Ax\|_{\infty}\ge S$, we have $r_{j}^{(t+1)}=\begin{cases}
r_{j}^{(t)} & j\neq i\\
r_{j}^{(t)}+1 & j=i
\end{cases}$, for a coordinate $i$ such that $|(Ax^{(t)})_{i}|=\left\Vert Ax^{(t)}\right\Vert _{\infty}$.
Let $v=\one_{i}$ be the indicator vector for coordinate $i$. We
have $r^{(t+1)}=r^{(t)}+v$ and $\|v\|_{1}=1$. Since $r^{(t)}\ge r^{(0)}>w$,
we can apply Lemma \ref{lem:energy-lewis-weight} and have 
\begin{align*}
\frac{\cE(r^{(t+1)})-\cE(r^{(t)})}{\|r^{(t+1)}\|_{1}-\|r^{(t)}\|_{1}} & \ge\frac{\frac{1}{2}v_{i}(Ax^{(t)})_{i}^{2}}{v_{i}}\ge\frac{S^{2}}{2}=\frac{d^{2/3}M^{2}}{2}\ge M^{2}.
\end{align*}
When $\|Ax\|_{\infty}<S$, for each coordinate $j$ such that $r_{j}^{(t+1)}>r_{j}^{(t)}$
we have $r_{j}^{(t+1)}=r_{j}^{(t)}\frac{(Ax^{(t)})_{j}^{2}}{M^{2}}$.
Hence,
\begin{align*}
\frac{(Ax^{(t)})_{j}^{2}r_{j}^{(t)}\big(1-\frac{r_{j}^{(t)}}{r_{j}^{(t+1)}}\big)}{r_{j}^{(t+1)}-r_{j}^{(t)}} & =\frac{(Ax^{(t)})_{j}^{2}r_{j}^{(t)}}{r_{j}^{(t+1)}}=M^{2}.
\end{align*}
 By Lemma \ref{lem:energy-usual}, 
\begin{align*}
\frac{\cE(r^{(t+1)})-\cE(r^{(t)})}{\|r^{(t+1)}\|_{1}-\|r^{(t)}\|_{1}} & \ge\frac{\sum_{j}(Ax^{(t)})_{j}^{2}r_{j}^{(t)}\big(1-\frac{r_{j}^{(t)}}{r_{j}^{(t+1)}}\big)}{\sum_{j}r_{j}^{(t+1)}-r_{j}^{(t)}}\ge M^{2}.
\end{align*}
\end{proof}
Lemma \ref{lem:Linf-invariant} allows us to show the quality of the
dual solution, if output. 
\begin{lem}
If Algorithm \ref{alg:Linf-regression} returns a dual solution $r^{(T)}$
then $\frac{\cE(r^{(T)})}{\|r^{(T)}\|_{1}}\ge\big(\frac{M}{1+\epsilon}\big)^{2}.$
\end{lem}
\begin{proof}
The case when the Algorithm terminates in Line 8 immediately gives
us the claim. We only have to care about the case when the while-loop
terminates with $\|r^{(T)}\|_{1}>\frac{\|r^{(0)}\|_{1}}{\epsilon}$.
By Lemma \ref{lem:Linf-invariant}, we have that
\begin{align*}
\sum_{i=0}^{T-1}\cE(r^{(t+1)})-\cE(r^{(t)}) & \ge M^{2}\sum_{i=0}^{T-1}\|r^{(t+1)}\|_{1}-\|r^{(t)}\|_{1}
\end{align*}
leading to $\cE(r^{(T)})\ge M^{2}(\|r^{(T)}\|_{1}-\|r^{(0)}\|_{1})$.
Since $\|r^{(T)}\|_{1}>\frac{\|r^{(0)}\|_{1}}{\epsilon}$
\begin{align*}
\frac{\cE(r^{(T)})}{\|r^{(T)}\|_{1}} & \ge M^{2}\bigg(1-\frac{\|r^{(0)}\|_{1}}{\|r^{(T)}\|_{1}}\bigg)\ge M^{2}\left(1-\epsilon\right)\ge\big(\frac{M}{1+\epsilon}\big)^{2},
\end{align*}
as needed.
\end{proof}

\subsection{Runtime}

To bound the runtime of Algorithm \ref{alg:Linf-regression}, we also
splits the iteration by the width of $\|Ax^{(t)}\|_{\infty}$. Let
$T_{hi}$ and $T_{lo}$ be the iterations when $\|Ax^{(t)}\|_{\infty}>S$
and $\|Ax^{(t)}\|_{\infty}\le S$, respectively before the algorithm
returns a solution. We also abuse the notations and denote by $T_{hi}$
and $T_{lo}$ the numbers of such iterations, respectively. The following
lemmas bound $T_{hi}$ and $T_{lo}$.
\begin{lem}
$T_{hi}\le\frac{6d^{1/3}}{\epsilon}.$
\end{lem}
\begin{proof}
Again, for $t$ such that $\left\Vert Ax^{(t)}\right\Vert _{\infty}>S$,
by Lemma \ref{lem:energy-lewis-weight}, we have
\begin{align*}
\frac{\cE(r^{(t+1)})-\cE(r^{(t)})}{\|r^{(t+1)}\|_{1}-\|r^{(t)}\|_{1}} & \ge\frac{\frac{1}{2}v_{i}(Ax^{(t)})_{i}^{2}}{v_{i}}\ge\frac{S^{2}}{2}=\frac{d^{2/3}M^{2}}{2}.
\end{align*}
Since $\|r^{(t+1)}\|_{1}-\|r^{(t)}\|_{1}=1$, we get $\cE(r^{(t+1)})-\cE(r^{(t)})\ge\frac{d^{2/3}M^{2}}{2}$.
Therefore $\cE(r^{(T)})\ge T_{hi}\cdot\frac{d^{2/3}M^{2}}{2}$. Notice
that $\|r^{(T)}\|_{1}\le\frac{\|r^{(0)}\|_{1}}{\epsilon}\le\frac{3d}{\epsilon}$
and $\frac{\cE(r^{(T)})}{\|r^{(T)}\|_{1}}\le\big(\frac{M}{1+\epsilon}\big)^{2}$,
we get that $T_{hi}\le\frac{6d^{1/3}}{\epsilon}.$
\end{proof}

\begin{lem}
\label{lem:Linf-low-width-iter}$T_{lo}\le O\big(\big(\frac{1}{\epsilon^{2}}+\frac{d^{\frac{1}{3}}}{\epsilon\log d}\big)\log\frac{n}{\epsilon}\big).$
\end{lem}
\begin{proof}
Let $\alpha_{j}^{(t)}=\frac{r_{j}^{(t+1)}}{r_{j}^{(t)}}$. Before
the algorithm terminates, we have
\begin{align*}
\left\Vert A\frac{\sum_{t\in T_{lo}}x^{(t)}}{T_{lo}}\right\Vert _{\infty} & \ge(1+\epsilon)M
\end{align*}
This means there must exist a coordinate $i$ such that $\sum_{t\in T_{lo}}\frac{|(Ax^{(t)})_{i}|}{M}\ge(1+\epsilon)T_{lo}$.
We have if $\frac{(Ax^{(t)})_{i}^{2}}{M^{2}}<\left(1+\epsilon\right)$,
then $\frac{|(Ax^{(t)})_{i}|}{M}<(1+\frac{\epsilon}{2})$, and if
$\frac{(Ax^{(t)})_{i}^{2}}{M^{2}}\ge\left(1+\epsilon\right)$, $\frac{(Ax^{(t)})_{i}^{2}}{M^{2}}=\alpha_{i}^{(t)}>1$.
Hence, 
\begin{align*}
\frac{|(Ax^{(t)})_{i}|}{M} & \le(1+\frac{\epsilon}{2})+\one_{\alpha_{i}^{(t)}>1}\sqrt{\alpha_{i}^{(t)}}
\end{align*}
We obtain 
\begin{align*}
T_{lo}(1+\frac{\epsilon}{2})+\sum_{t\in T_{lo},\alpha_{i}^{(t)}>1}\sqrt{\alpha_{i}^{(t)}} & \ge(1+\epsilon)T_{lo}
\end{align*}
and thus $\sum_{t\in T_{lo},\alpha_{i}^{(t)}>1}\sqrt{\alpha_{i}^{(t)}}\ge\frac{\epsilon T_{lo}}{2}$.
For $t$ such that $\alpha_{i}^{(t)}>1$, we have $\sqrt{\alpha_{i}^{(t)}}=\frac{|(Ax^{(t)})_{i}|}{M}\in[(1+\epsilon)^{\frac{1}{2}},d^{\frac{1}{3}}]$.
By Lemma A.1 from \cite{ene2019improved}, we have a lower bound for
$\prod_{t\in T_{lo},\alpha_{i}^{(t)}>1}\alpha_{i}^{(t)}$:
\begin{align}
\prod_{t\in T_{lo},\alpha_{i}^{(t)}>1}\alpha_{i}^{(t)} & \ge\min\left\{ \left(1+\epsilon\right)^{\frac{\epsilon T_{lo}}{2\left(1+\epsilon\right)^{\frac{1}{2}}}},d^{\frac{2}{3}\frac{\epsilon T_{lo}}{2d^{\frac{1}{3}}}}\right\} \label{eq:low}
\end{align}
On the other hand, $\frac{r_{i}^{(T)}}{r_{i}^{(0)}}$ is an upperbound
for $\prod_{t\in T_{lo},\alpha_{i}^{(t)}>1}\alpha_{i}^{(t)}$. We
initialize $r_{i}^{(0)}\ge\frac{d}{n}$ and before the algorithm terminates
$r_{i}^{(T)}\le\|r^{(T)}\|_{1}\le\frac{w+d}{\epsilon}\le\frac{3d}{\epsilon}$.
We now have that
\begin{align}
\prod_{t\in T_{lo},\alpha_{i}^{(t)}>1}\alpha_{i}^{(t)} & \le\frac{r_{i}^{(T)}}{r_{i}^{(0)}}\le\frac{3d}{\epsilon}\cdot\frac{n}{d}\le\frac{3n}{\epsilon}\label{eq:high}
\end{align}
From (\ref{eq:low}) and (\ref{eq:high}), we obtain $T_{lo}\le O\big(\big(\frac{1}{\epsilon^{2}}+\frac{d^{\frac{1}{3}}}{\epsilon\log d}\big)\log\frac{3n}{\epsilon}\big).$
\end{proof}

\section{Analysis of the QSC Algorithm\protect\label{sec:analysis-QSC-alg}}

In this section, we analyze our main Algorithm \ref{alg:QSC-main}
for QSC optimization. The algorithm starts with a solution $x^{(0)}$
and iteratively updates it. Building on prior work \cite{cohen2017matrix,allen2017much,adil2021unifying},
we show that the algorithm returns a nearly-optimal solution provided
we have access to an algorithm for solving the following residual
problems. Letting $x^{(t)}$ be the solution in iteration $t$ of
Algorithm \ref{alg:QSC-main}, we would like to find an approximate
solution to the following residual problem:
\begin{align}
\max_{\Delta\colon\|A\Delta\|_{\infty}\le\frac{1}{C}}\res_{x^{(t)}}(\Delta) & \coloneqq\nabla f(x^{(t)})^{\top}(A\Delta)-\frac{1}{e}(A\Delta)^{\top}\nabla^{2}f(x^{(t)})(A\Delta)\label{eq:residual-problem-2}
\end{align}
Using a binary search approach as in \cite{adil2021unifying}, we
can find a value $M$ that is a $2$-approximation to the optimal
value of the above residual problem. This binary search is shown in
the two for-loops in Line \ref{forloop-nu} and Line \ref{forloop-M}
of Algorithm \ref{alg:QSC-main}. With each guess $M$, the algorithm
uses the $\ressolver$ (Algorithm \ref{alg:residual-solver}) to find
a solution to the residual problem. We show that if the residual solver
returns a constant factor approximation to the residual problem then
Algorithm \ref{alg:QSC-main} returns a nearly-optimal solution. We
now give the formal analysis.

We start by showing that the steps executed in the two for-loops in
Algorithm \ref{alg:QSC-main} are sufficient to find a sufficiently
close guess for the residual problem objective. The lemma follows
from Lemmas 4.3, 4.4 in \cite{adil2021unifying}.
\begin{lem}[Lemma 4.3, 4.4 \cite{adil2021unifying}]
\label{lem:bin-search-opt-1} For each value $t$ of the outer-most
iteration with iterate $x^{(t)}$, there is an inner-most iteration
of Algorithm \ref{alg:QSC-main} for which the following hold. Let
$M$ be the value considered in that iteration. Consider the residual
problem:
\[
\max_{\Delta\colon\|A\Delta\|_{\infty}\le1/C}\res_{x^{(t)}}(\Delta)=\nabla f(x^{(t)})^{\top}(A\Delta)-\frac{1}{e}(A\Delta)^{\top}\nabla^{2}f(x^{(t)})(A\Delta)
\]
Let $\Delta^{*}\in\arg\max_{\|A\Delta\|_{\infty}\le1/C}\res_{x^{(t)}}(\Delta)$
be an optimal solution to the residual problem, and let $\opt=\res_{x^{(t)}}(\Delta^{*})$
be its objective value. We have $\opt\in(\frac{M}{2},M]$ and $(A\Delta^{*})^{\top}\nabla^{2}f(x^{(t)})(A\Delta^{*})\le eM$.
\end{lem}
Next, we show that $\ressolver(x^{(t)},M)$ returns a constant factor
approximation to the residual problem $\max_{\Delta\colon\|A\Delta\|_{\infty}\le1/C}\res_{x^{(t)}}(\Delta)$
when we run it with the value $M$ guaranteed by the above lemma.
We will provide the proof of this lemma in Section \ref{subsec:residual-solver-guarantee}. 
\begin{lem}
\label{lem:residual-solver-guarantee} Consider an outermost iteration
$t$ of Algorithm \ref{alg:QSC-main} and an innermost iteration with
a value $M$ with the properties guaranteed by Lemma \ref{lem:bin-search-opt-1}.
Let $\widetilde{\Delta}\coloneqq\Delta_{\nu,M}$ be the solution returned
by $\ressolver(x^{(t)},M)$. We have $\|A\widetilde{\Delta}\|_{\infty}\le\frac{1}{C}$
and $\res_{x^{(t)}}(\widetilde{\Delta})\ge\frac{\opt}{20}.$
\end{lem}
Equipped with the above guarantee, we can show that the algorithm
converges in a small number of iterations. To this end, we first prove
the following lemma which shows that each iteration significantly
reduces the optimality gap.
\begin{lem}
\label{lem:progress-1} Let $x$ be an iterate. Suppose that $\ressolver(x,M)$
returns a solution $\widetilde{\Delta}$ satisfying $\|A\widetilde{\Delta}\|_{\infty}\le\frac{1}{C}$
and $\res_{x}(\widetilde{\Delta})\ge\kappa\max_{\|A\Delta\|_{\infty}\le\frac{1}{C}}\res_{x}(\Delta)$.
We have
\begin{align*}
h\Big(x-\frac{\widetilde{\Delta}}{e^{2}}\Big)-h(x^{*}) & \le\Big(1-\frac{\kappa}{e^{2}CR}\Big)\left(h(x)-h(x^{*})\right).
\end{align*}
\end{lem}
In order to prove Lemma \ref{lem:progress-1}, first we recall the
notion of Hessian stability.
\begin{defn}
A function $h:\R^{d}\to\R$ is $(r,d(r))$-Hessian stable in the $\ell_{\infty}$-norm
iff for all $x,y$ such that $\left\Vert x-y\right\Vert _{\infty}\le r$
we have 
\begin{align*}
\frac{1}{d(r)}\nabla^{2}h(x) & \preceq\nabla^{2}h(y)\preceq d(r)\nabla^{2}h(x).
\end{align*}
\end{defn}
\begin{fact}[From \cite{carmon2020acceleration}]
For a $C$-quasi-self-concordant function $f$, $\sum_{i}f(x_{i})$
is $\left(r,e^{Cr}\right)$-Hessian stable in the $\ell_{\infty}$-norm.
\end{fact}
We now give the proof of Lemma \ref{lem:progress-1}.

\begin{proof}
Let $x^{*}=\arg\min h\left(x\right)$, we have $\left\Vert Ax-Ax^{*}\right\Vert _{\infty}\le R$,
and $h$ is $\left(r,e\right)$-Hessian stable where $r=\frac{1}{C}$.
Let $\widehat{x}=\frac{r}{R}x^{*}+\left(1-\frac{r}{R}\right)x$; and
let $\widehat{\Delta}=x-\widehat{x}=\frac{r}{R}\left(x-x^{*}\right)$.
We have $\left\Vert A\widehat{\Delta}\right\Vert _{\infty}\le r$.
Suppose that $\Delta^{*}$ is an optimal solution to the residual
problem. Also recall the notation $\nabla h(x)=A^{\top}\nabla f(x)$
and $\nabla^{2}h(x)=A^{\top}\nabla^{2}f(x)A$ where we write $\nabla f(x)=\Bigg(\begin{array}{c}
f'((Ax-b)_{1})\\
\dots\\
f'((Ax-b)_{n})
\end{array}\Bigg)$, $\nabla^{2}f(x)=\diag(f''((Ax-b)_{1}),\dots,f''((Ax-b)_{n}))$.

Let $k=e^{2}$ and using hessian stability, we have
\begin{align*}
h\left(x-\widehat{\Delta}\right)-h\left(x\right) & \ge-\nabla f\left(x\right)^{\top}\left(A\widehat{\Delta}\right)+\frac{1}{\exp\left(Cr\right)}\left(A\widehat{\Delta}\right)^{\top}\nabla^{2}f\left(x\right)\left(A\widehat{\Delta}\right)\\
 & =-\nabla f\left(x\right)^{\top}\left(A\widehat{\Delta}\right)+\frac{1}{e}\left(A\widehat{\Delta}\right)^{\top}\nabla^{2}f\left(x\right)\left(A\widehat{\Delta}\right)\\
 & =-\res_{x}\left(\widehat{\Delta}\right);\\
h\left(x-\frac{\widetilde{\Delta}}{k}\right)-h\left(x\right) & \le-\nabla f\left(x\right)^{\top}\left(\frac{A\widetilde{\Delta}}{k}\right)+\exp\left(\frac{Cr}{k}\right)\left(\frac{A\widetilde{\Delta}}{k}\right)^{\top}\nabla^{2}g\left(x\right)\left(\frac{A\widetilde{\Delta}}{k}\right)\\
 & \le\frac{1}{k}\left(-\nabla f\left(x\right)^{\top}\left(A\widetilde{\Delta}\right)+\frac{1}{e}\left(A\widetilde{\Delta}\right)^{\top}\nabla^{2}f\left(x\right)\left(A\widetilde{\Delta}\right)\right)\\
 & =-\frac{1}{k}\res_{x}\left(\widetilde{\Delta}\right).
\end{align*}
Since $\res\left(\widetilde{\Delta}\right)\ge\kappa\res_{x}\left(\Delta^{*}\right)\ge\kappa\res_{x}\left(\widehat{\Delta}\right)$,
we have
\begin{align*}
h\left(x-\frac{\widetilde{\Delta}}{k}\right)-h\left(x\right) & \le-\frac{\kappa}{k}\res_{x}\left(\widehat{\Delta}\right)\\
 & \le\frac{\kappa}{k}\left(h\left(x-\widehat{\Delta}\right)-h\left(x\right)\right)\\
 & =\frac{\kappa}{k}\left(h\left(\widehat{x}\right)-h\left(x\right)\right)\\
 & \le\frac{\kappa r}{kR}\left(h\left(x^{*}\right)-h\left(x\right)\right)
\end{align*}
This give us the conclusion
\begin{align*}
h\left(x-\frac{\widetilde{\Delta}}{k}\right)-h\left(x^{*}\right) & \le\left(1-\frac{\kappa r}{kR}\right)\left(h\left(x\right)-h\left(x^{*}\right)\right).
\end{align*}
\end{proof}

By combining Lemmas \ref{lem:residual-solver-guarantee} and \ref{lem:progress-1},
we obtain the following convergence guarantee.
\begin{lem}
\label{lem:qsc-main-convergence}Algorithm \ref{alg:QSC-main} constructs
a solution $x^{(T)}$ such that $h(x^{(T)})\le h(x^{*})+\epsilon$
using $T=O\left(RC\log\left(\frac{h(x^{(0)})-h(x^{*})}{\epsilon}\right)\right)$
iterations of the outermost loop.
\end{lem}
Next, we analyze the overall running time of Algorithm \ref{alg:QSC-main}.
We have the following upper bound on number of calls that Algorithm
\ref{alg:QSC-main} makes to the residual solver.
\begin{lem}
\label{lem:qsc-main-calls-to-residual}In each iteration $t$, Algorithm
\ref{alg:QSC-main} makes $O\left(\log\left(CR\right)\log\left(\frac{h(x^{(0)})-h(x^{*})}{\epsilon}\right)\right)$
calls to the residual solver \ref{alg:residual-solver}.
\end{lem}
\begin{proof}
The algorithm executes a binary search using two for-loops \ref{forloop-nu}
and \ref{forloop-M}. The number of steps of for-loop \ref{forloop-nu}
is $O\left(\log\frac{h(x^{(t)})-B}{\epsilon}\right)=O\left(\log\frac{h(x^{(0)})-h(x^{*})}{\epsilon}\right)$
and the number of steps of for-loop \ref{forloop-M} is $O\left(\log CR\right)$,
giving us the claim.
\end{proof}

In Section \ref{lem:residual-solver-runtime}, we show the following
upper bound on the running time of the $\ressolver$.
\begin{lem}
\label{lem:residual-solver-runtime}$\ressolver$ uses $O(d^{1/3}\log n)$
linear system solves.
\end{lem}
Finally, to conclude the proof of Theorem \ref{thm:QSC-main-theorem},
we only have to combine Lemmas \ref{lem:qsc-main-convergence}, \ref{lem:qsc-main-calls-to-residual},
\ref{lem:residual-solver-runtime}.

\subsection{Proof of Lemma \ref{lem:residual-solver-guarantee}\protect\label{subsec:residual-solver-guarantee}}

In this section, we prove the approximation guarantee of the $\ressolver$
stated in Lemma \ref{lem:residual-solver-guarantee}. Let $x$ and
$M$ be the input to the $\ressolver$. As before, we let $\Delta^{*}\in\max_{\Delta\colon\|A\Delta\|_{\infty}\le1/C}\res_{x}(\Delta)$
and $\opt=\res_{x}(\Delta^{*})$ . As in Section \ref{sec:Generalization-to-QSC},
we let $g_{i}=\frac{-1}{M}\left(A^{\top}\nabla f(x)\right)_{i}$ and
$s_{i}=f''\left((Ax)_{i}\right)$ for all $i$.

We start with the following lemma.
\begin{lem}
\label{lem:solution-guarantee}Consider an iterate $x$ and a guess
$M$. Consider an algorithm that takes as input $x$ and $M$ and
it outputs a solution to the Problem \ref{eq:new-residual-problem}.
If $\opt\in(\frac{M}{2},M]$ and the algorithm satisfies Invariant
\ref{inv:invariant1}, then the algorithm outputs a solution $\Delta$
such that for $\widehat{\Delta}=\frac{\Delta}{11},$we have $\|A\widehat{\Delta}\|_{\infty}\le\frac{1}{C}$
and $\res_{x}(\widehat{\Delta})\ge\frac{\opt}{20}.$
\end{lem}
\begin{proof}
Let $a=-g^{\top}\Delta^{*}$. Note that $\frac{\Delta^{*}}{a}$ satisfies
$g^{\top}\frac{\Delta^{*}}{a}=-1$. Since $\opt\in(\frac{M}{2},M]$,
we have
\begin{align*}
a & =\frac{\nabla f\left(x\right)^{\top}\left(A\Delta^{*}\right)}{M}\ge\frac{\opt}{M}\ge\frac{1}{2}.
\end{align*}
Moreover since $\left\Vert A\Delta^{*}\right\Vert _{\infty}\le\frac{1}{C}$
and from Lemma \ref{lem:bin-search-opt}, we have 
\begin{align*}
\left\langle s,\left(A\frac{\Delta^{*}}{a}\right)^{2}\right\rangle +\frac{MC^{2}}{2}\left\Vert A\frac{\Delta^{*}}{a}\right\Vert _{\infty}^{2} & \le4eM+2M<13M.
\end{align*}
This means, our algorithm will output a solution $\Delta$ that satisfies
\begin{align*}
\nabla f\left(x\right)^{\top}\left(A\Delta\right) & =M\\
\left\Vert A\Delta\right\Vert _{\infty} & \le\frac{11}{C}\\
\left\langle s,\left(A\Delta\right)^{2}\right\rangle  & \le\min_{g^{\top}\Delta=-1}\left(\left\langle s,\left(A\Delta\right)^{2}\right\rangle +\frac{MC^{2}}{2}\left\Vert A\Delta\right\Vert _{\infty}^{2}\right)<13M
\end{align*}
Let $\widehat{\Delta}=\frac{\Delta}{11},$we have 
\begin{align*}
\nabla f\left(x\right)^{\top}\left(A\widehat{\Delta}\right) & =\frac{M}{11}\\
\left\Vert A\widehat{\Delta}\right\Vert _{\infty} & \le\frac{1}{C}\\
\left\langle s,\left(A\widehat{\Delta}\right)^{2}\right\rangle  & <\frac{13}{121}M
\end{align*}
which gives us 
\begin{align*}
\res_{x}\left(\widehat{\Delta}\right) & \ge\frac{M}{11}-\frac{1}{e}\cdot\frac{13}{121}M>\frac{M}{20}\ge\frac{\opt}{20}.
\end{align*}
\end{proof}

Thus it only remains to show that $\ressolver$ satisfies the aforementioned
invariant. Notice that the Problem \ref{eq:new-residual-problem}
has a similar structure as an $\ell_{\infty}$-regression problem
solved in the previous section, albeit with an additional quadratic
term $\left\langle s,(A\Delta)^{2}\right\rangle $. Fortunately, the
algorithm only needs to return a constant factor approximation, instead
of a $1+\epsilon$ approximation. This allows us to extend the approach
we used for Algorithm \ref{alg:Linf-regression}, and use a similar
analysis.

In the remainder of this section, we show the following lemma, using
similar ideas as in our analysis of our $\ell_{\infty}$ regression
algorithm.
\begin{lem}
\label{lem:residual-solver-invariant}Algorithm \ref{alg:residual-solver}
satisfies Invariant \ref{inv:invariant1}.
\end{lem}
By combining Lemmas \ref{lem:solution-guarantee} and \ref{lem:residual-solver-invariant},
we obtain Lemma \ref{lem:main-residual-solver-lemma}.

We proceed with the proof Lemma \ref{lem:residual-solver-invariant}
by showing the following lemmas.
\begin{lem}
\label{lem:quadratic-term-guarantee}For all iterations $t\ge1$ of
$\ressolver$, we have 
\begin{align*}
\left\langle s,\left(A\Delta^{(t)}\right)^{2}\right\rangle  & \le\min_{g^{\top}\Delta=-1}\left\langle s,\left(A\Delta\right)^{2}\right\rangle +\frac{MC^{2}}{2}\left\Vert A\Delta\right\Vert _{\infty}^{2}.
\end{align*}
\end{lem}
\begin{proof}
Let $\Delta^{*}=\arg\min_{g^{\top}\Delta=-1}\left\langle s,\left(A\Delta\right)^{2}\right\rangle +\frac{MC^{2}}{2}\left\Vert A\Delta\right\Vert _{\infty}^{2}$.
We have 
\begin{align*}
 & 2\left(\left\Vert w\right\Vert _{1}+d\right)\left\langle s,\left(A\Delta^{(t)}\right)^{2}\right\rangle \\
 & \le\left\langle \frac{MC^{2}}{2}r^{(t)}+2\left(\left\Vert w\right\Vert _{1}+d\right)s,\left(A\Delta^{(t)}\right)\right\rangle \\
 & \le\left\langle \frac{MC^{2}}{2}r^{(t)}+2\left(\left\Vert w\right\Vert _{1}+d\right)s,\left(A\Delta^{*}\right)\right\rangle \qquad\text{due to the optimality of }x^{(t)}\\
 & \le2\left(\left\Vert w\right\Vert _{1}+d\right)\left\langle \frac{MC^{2}}{2}\frac{r^{(t)}}{\left\Vert r^{(t)}\right\Vert _{1}}+s,\left(A\Delta^{*}\right)^{2}\right\rangle \qquad\text{since }\left\Vert r^{(t)}\right\Vert _{1}\le2\left(\left\Vert w\right\Vert _{1}+d\right)
\end{align*}
which gives
\begin{align*}
\left\langle s,\left(A\Delta^{(t)}\right)^{2}\right\rangle  & \le\left\langle \frac{MC^{2}}{2}\frac{r^{(t)}}{\left\Vert r^{(t)}\right\Vert _{1}}+s,\left(A\Delta^{*}\right)^{2}\right\rangle \\
 & \le\left\langle s,\left(A\Delta^{*}\right)^{2}\right\rangle +\frac{MC^{2}}{2}\left\Vert A\Delta^{*}\right\Vert _{\infty}^{2},
\end{align*}
as needed.
\end{proof}

\begin{lem}
\label{lem:res-invariant}$\ressolver$ maintains the invariant that,
for all $t\ge0$,
\begin{align*}
\frac{\cE\left(p^{(t+1)}\right)-\cE\left(p^{(t)}\right)}{\left\Vert r^{(t+1)}\right\Vert _{1}-\left\Vert r^{(t)}\right\Vert _{1}} & \ge26M.
\end{align*}
\end{lem}
\begin{proof}
When $\left\Vert A\Delta^{(t)}\right\Vert _{\infty}\ge S=\frac{11d^{\frac{1}{3}}}{C}$,
we have $r_{j}^{(t+1)}=\begin{cases}
r_{j}^{(t)} & j\neq i\\
r_{j}^{(t)}+1 & j=i
\end{cases}$, for a coordinate $i$ such that $\left(A\Delta^{(t)}\right)_{i}=\left\Vert A\Delta^{(t)}\right\Vert _{\infty}$.
Let $v=\one_{i}$ be the indicator vector for coordinate $i$. We
have $r^{(t+1)}=r^{(t)}+v$ and $\left\Vert v\right\Vert _{1}=1$.
Using Lemma \ref{lem:energy-lewis-weight} we have 
\begin{align*}
\frac{\cE\left(p^{(t+1)}\right)-\cE\left(p^{(t)}\right)}{\left\Vert r^{(t+1)}\right\Vert _{1}-\left\Vert r^{(t)}\right\Vert _{1}} & \ge\frac{\frac{MC^{2}}{2}\cdot\frac{1}{2}v_{i}\left(A\Delta^{(t)}\right)_{i}^{2}}{v_{i}}\ge\frac{MC^{2}S^{2}}{4}=26Md^{\frac{2}{3}}\ge26M.
\end{align*}
When $\left\Vert A\Delta^{(t)}\right\Vert _{\infty}<S$, for each
coordinate $j$ such that $r_{j}^{(t+1)}>r_{j}^{(t)}$ we have $r_{j}^{(t+1)}=\frac{1}{52}r_{j}^{(t)}\left(A\Delta^{(t)}\right)_{j}^{2}C^{2}$.
Note that $\frac{p^{(t)}}{p^{(t+1)}}\ge\frac{r_{j}^{(t)}}{r_{j}^{(t+1)}}$.
Hence,
\begin{align*}
\frac{\left(A\Delta^{(t)}\right)_{j}^{2}p_{j}^{(t)}\left(1-\frac{p_{j}^{(t)}}{p_{j}^{(t+1)}}\right)}{r_{j}^{(t+1)}-r_{j}^{(t)}} & =\frac{\left(A\Delta^{(t)}\right)_{j}^{2}\frac{p_{j}^{(t)}}{p_{j}^{(t+1)}}\left(p_{j}^{(t+1)}-p_{j}^{(t)}\right)}{r_{j}^{(t+1)}-r_{j}^{(t)}}\\
 & \ge\frac{\frac{MC^{2}}{2}\left(A\Delta^{(t)}\right)_{j}^{2}\frac{r_{j}^{(t)}}{r_{j}^{(t+1)}}\left(r_{j}^{(t+1)}-r_{j}^{(t)}\right)}{r_{j}^{(t+1)}-r_{j}^{(t)}}\\
 & =26M.
\end{align*}
 by Lemma \ref{lem:energy-usual}, 
\begin{align*}
\frac{\cE\left(p^{(t+1)}\right)-\cE\left(p^{(t)}\right)}{\left\Vert r^{(t+1)}\right\Vert _{1}-\left\Vert r^{(t)}\right\Vert _{1}} & \ge\frac{\sum_{j}\left(A\Delta^{(t)}\right)_{j}^{2}p_{j}^{(t)}\left(1-\frac{p_{j}^{(t)}}{p_{j}^{(t+1)}}\right)}{\sum_{j}\left(r_{j}^{(t+1)}-r_{j}^{(t)}\right)}\\
 & \ge\frac{\sum_{j:r_{j}^{(t+1)}>r_{j}^{(t)}}\left(A\Delta^{(t)}\right)_{j}^{2}p_{j}^{(t)}\left(1-\frac{p_{j}^{(t)}}{p_{j}^{(t+1)}}\right)}{\sum_{j:r_{j}^{(t+1)}>r_{j}^{(t)}}\left(r_{j}^{(t+1)}-r_{j}^{(t)}\right)}\\
 & \ge26M.
\end{align*}
\end{proof}

\begin{lem}
If $\ressolver$ returns a dual solution $r^{(T)}$ then $\cE\left(s+\frac{MC^{2}}{2\left\Vert r^{(T)}\right\Vert _{1}}r^{(T)}\right)\ge13M.$
\end{lem}
\begin{proof}
By Lemma \ref{lem:res-invariant}, we have that
\begin{align*}
\sum_{i=0}^{T-1}\cE\left(p^{(t+1)}\right)-\cE\left(p^{(t)}\right) & \ge26M\sum_{i=0}^{T-1}\left\Vert r^{(t+1)}\right\Vert _{1}-\left\Vert r^{(t)}\right\Vert _{1}
\end{align*}
leading to 
\begin{align*}
\cE(p^{(T)}) & \ge26M\left(\left\Vert r^{(T)}\right\Vert _{1}-\left\Vert r^{(0)}\right\Vert _{1}\right)
\end{align*}
Since $\left\Vert r^{(T)}\right\Vert _{1}\ge2\left(\left\Vert w\right\Vert _{1}+d\right)$
\begin{align*}
\cE\left(s+\frac{MC^{2}}{2\left\Vert r^{(T)}\right\Vert _{1}}r^{(T)}\right) & \ge\cE\left(\frac{2\left(\left\Vert w\right\Vert _{1}+d\right)s}{\left\Vert r^{(T)}\right\Vert _{1}}+\frac{MC^{2}}{2\left\Vert r^{(T)}\right\Vert _{1}}r^{(T)}\right)\\
 & \ge26M\left(1-\frac{\left\Vert r^{(0)}\right\Vert _{1}}{\left\Vert r^{(T)}\right\Vert _{1}}\right)\\
 & \ge13M,
\end{align*}
as needed.
\end{proof}

\begin{lem}
\label{lem:primal-output}If $\ressolver$ returns a primal solution
$\Delta$ then $\left\Vert A\Delta\right\Vert _{\infty}\le\frac{11}{C}$
and
\begin{align*}
\left\langle s,\left(A\Delta\right)^{2}\right\rangle  & \le\min_{g^{\top}\Delta=-1}\left\langle s,\left(A\Delta\right)^{2}\right\rangle +\frac{MC^{2}}{2}\left\Vert A\Delta\right\Vert _{\infty}^{2}.
\end{align*}
\begin{proof}
If the algorithm returns $\Delta$, we have $\left\Vert A\Delta\right\Vert _{\infty}\le\frac{11}{C}$,
either by a solution in a single iteration, or the average over iterations
when $\left\Vert A\Delta^{(t)}\right\Vert _{\infty}\le S$. In both
case, by Lemma \ref{lem:quadratic-term-guarantee}, and convexity,
we have
\begin{align*}
\left\langle s,\left(A\Delta\right)^{2}\right\rangle  & \le\min_{g^{\top}\Delta=-1}\left\langle s,\left(A\Delta\right)^{2}\right\rangle +\frac{MC^{2}}{2}\left\Vert A\Delta\right\Vert _{\infty}^{2}.
\end{align*}
\end{proof}
\end{lem}

\subsection{Proof of Lemma \ref{lem:residual-solver-runtime}\protect\label{subsec:residual-solver-runtime}}

In this section, we analyze the running time of the $\ressolver$
(Algorithm \ref{alg:residual-solver}). The running time is dominated
by the linear system solves, and thus we upper bound the number of
calls to the linear system solver. 

We proceed similarly to the analysis of our $\ell_{\infty}$ regression
algorithm. Let $T_{hi}$ and $T_{lo}$be the iterations when $\left\Vert A\Delta^{(t)}\right\Vert _{\infty}>S$
and $\left\Vert A\Delta^{(t)}\right\Vert _{\infty}\le S$, respectively
before the algorithm returns a primal solution or a dual one. We also
abuse the notations and denote by $T_{hi}$ and $T_{lo}$ the numbers
of such iterations, respectively.
\begin{lem}
$T_{hi}\le O\left(d^{1/3}\right).$
\end{lem}
\begin{proof}
Again, for $t$ such that $\left\Vert A\Delta^{(t)}\right\Vert _{\infty}>S$,
by Lemma \ref{lem:energy-lewis-weight}, we have
\begin{align*}
\frac{\cE\left(p^{(t+1)}\right)-\cE\left(p^{(t)}\right)}{\left\Vert r^{(t+1)}\right\Vert _{1}-\left\Vert r^{(t)}\right\Vert _{1}} & \ge26Md^{\frac{2}{3}}.
\end{align*}
Since $\left\Vert r^{(t+1)}\right\Vert _{1}-\left\Vert r^{(t)}\right\Vert _{1}=1$,
we get 
\begin{align*}
\cE\left(p^{(t+1)}\right)-\cE\left(p^{(t)}\right) & \ge26Md^{\frac{2}{3}}
\end{align*}
Therefore 
\begin{align*}
\cE\left(2\left(\left\Vert w\right\Vert _{1}+d\right)s+\frac{MC^{2}}{2}r^{(T)}\right) & \ge26Md^{\frac{2}{3}}T_{hi}
\end{align*}
Notice that $\left\Vert r^{(T)}\right\Vert _{1}\le2\left(\left\Vert w\right\Vert _{1}+d\right)\le6d$
and $\cE\left(s+\frac{MC^{2}}{2}\frac{r^{(T)}}{\left\Vert r^{(T)}\right\Vert _{1}}\right)\le13M\left\Vert r^{(T)}\right\Vert _{1}$
\begin{align*}
\cE\left(2\left(\left\Vert w\right\Vert _{1}+d\right)s+\frac{MC^{2}}{2}r^{(T)}\right) & =2\left(\left\Vert w\right\Vert _{1}+d\right)\cE\left(s+\frac{MC^{2}}{2}\frac{r^{(T)}}{2\left(\left\Vert w\right\Vert _{1}+d\right)}\right)\\
 & \le6d\cE\left(s+\frac{MC^{2}}{2}\frac{r^{(T)}}{\left\Vert r^{(T)}\right\Vert _{1}}\right)\\
 & \le78dM.
\end{align*}
We obtain $T_{hi}\le3d^{1/3}.$
\end{proof}

\begin{lem}
$T_{lo}\le O\left(d^{\frac{1}{3}}\log\left(n\right)\right).$
\end{lem}
\begin{proof}
Let $\alpha_{j}^{(t)}=\frac{r_{j}^{(t+1)}}{r_{j}^{(t)}}$. We have
\begin{align*}
\left\Vert A\frac{\sum_{t\in T_{lo}}\Delta^{(t)}}{T_{lo}}\right\Vert _{\infty} & \ge\frac{11}{C}
\end{align*}
We obtain that there exists a coordinate $i$ such that 
\begin{align*}
\sum_{t\in T_{lo}}\left(A\Delta^{(t)}\right)_{i}C & \ge11T_{lo}
\end{align*}
We have if $\left(A\Delta^{(t)}\right)_{i}^{2}C^{2}<100$, $\left(A\Delta^{(t)}\right)_{i}C<10$,
and if $\left(A\Delta^{(t)}\right)_{i}^{2}C^{2}\ge100$, $\left(A\Delta^{(t)}\right)_{i}^{2}C^{2}=52\alpha_{i}^{(t)}$.
Hence, 
\begin{align*}
\left(A\Delta^{(t)}\right)_{i}C & \le10+\one_{\alpha_{i}^{(t)}>1}\sqrt{52\alpha_{i}^{(t)}}
\end{align*}
We obtain 
\begin{align*}
\sum_{t\in T_{lo},\alpha_{i}^{(t)}>1}\sqrt{\alpha_{i}^{(t)}} & \ge\frac{T_{lo}}{\sqrt{52}}
\end{align*}
For $t$ such that $\alpha_{i}^{(t)}>1$, we have
\begin{align*}
\sqrt{\alpha_{i}^{(t)}} & =\frac{1}{\sqrt{52}}\left(A\Delta^{(t)}\right)_{i}C\in\left[\frac{10}{\sqrt{52}},\frac{11d^{\frac{1}{3}}}{\sqrt{52}}\right]
\end{align*}
Similarly to Lemma \ref{lem:Linf-low-width-iter}, we obtain 
\begin{align*}
T_{lo} & \le O\left(d^{\frac{1}{3}}\log\left(n\right)\right).
\end{align*}
\end{proof}

Lemma \ref{lem:residual-solver-runtime} now follows from the above
lemmas.

\section{QSC Algorithm for the Underdetermined Case}

In the underdetermined case $n\le d$, instead of using Lewis weights
in the residual solver, we can simply return to the algorithm by \cite{ene2019improved}
and use a uniform initialization of the resistances. We provide the
algorithm in Algorithm \ref{alg:residual-solver-1}. The analysis
follows similarly (which we omit). The number of linear system solves
is $O(n^{1/3}\log n)$.

\begin{algorithm}[th]
\caption{$\protect\ressolver(x,M)$ for $n\le d$}

\label{alg:residual-solver-1}

\begin{algorithmic}[1]

\STATE\textbf{Initialize}: $r^{(0)}=1$

\STATE$g_{i}=\frac{-1}{M}(A^{\top}\nabla f(x))_{i}$, $s_{i}=f''((Ax)_{i})$,
$t=0$, $t'=0$, $v^{(t')}=0$

\STATE\textbf{while} $\|r^{(t)}\|_{1}\leq2n$

\STATE$\qquad$Let $p^{(t)}=2ns+\frac{MC^{2}}{2}r^{(t)}$; $\Delta^{(t)}=\arg\min_{\Delta:g^{\top}\Delta=-1}\left\langle p^{(t)},(A\Delta)^{2}\right\rangle $

\STATE$\qquad$\textbf{if} $\left\langle s+\frac{MC^{2}}{2}\frac{r^{(t)}}{\left\Vert r^{(t)}\right\Vert _{1}},(A\Delta^{(t)})^{2}\right\rangle \ge13M$
\textbf{then} \textbf{return} $r^{(t)}$

\STATE$\qquad$\textbf{if} $\|A\Delta^{(t)}\|_{\infty}\le\frac{11}{C}$
\textbf{then} \textbf{return} $\Delta^{(t)}$

\STATE$\qquad$\textbf{if} $\|A\Delta^{(t)}\|_{\infty}\le S=\frac{11n^{\frac{1}{3}}}{C}$:

\STATE$\qquad\qquad$Let $t'=t'+1$; $v^{(t')}=v^{(t'-1)}+\Delta^{(t)}$

\STATE$\qquad\qquad$\textbf{if} $\|Av^{(t')}\|_{\infty}/t'\le\frac{11}{C}$
\textbf{then} \textbf{return} $v^{(t')}/t'$

\STATE$\qquad\qquad$$r_{j}^{(t+1)}=\begin{cases}
\frac{1}{52}r_{j}^{(t)}(A\Delta^{(t)})_{j}^{2}C^{2} & \text{if }(A\Delta^{(t)})_{j}^{2}\ge\frac{100}{C^{2}}\\
r_{j}^{(t)} & \mbox{otherwise}
\end{cases}$

\STATE$\qquad$$t=t+1$

\STATE\textbf{return} $r^{(t)}$

\end{algorithmic}
\end{algorithm}

\section{Approximating the Lewis Weights\protect\label{sec:approx-lewis}}

Here we review the fixed point iteration by \cite{cohen2019near}
for computing approximate $\ell_{\infty}$ Lewis weights. While the
algorithm is known to provide only a one-sided approximation, this
guarantee suffices for our application.
\begin{thm}
On input $A\in\mathbb{R}^{n\times d}$, the algorithm $\mathsf{ApproxLewis}(A)$
returns w.h.p. $\ell_{\infty}$ Lewis weights overestimates $w\in\mathbb{R}^{n}$
in the sense that
\begin{align*}
w_{i} & \geq1_{i}^{\top}W^{1/2}A\left(A^{\top}WA\right)^{-1}A^{\top}W^{1/2}1_{i}\,,\\
d & \leq\left\Vert w\right\Vert _{1}\leq2d\,,
\end{align*}
in time $O\left(\log n\cdot\mathcal{T}_{A}\right)$, where $\mathcal{T}_{A}$
is the time to solve a linear system involving a matrix $A^{\top}DA$,
where $D$ is a positive diagonal.
\end{thm}
\begin{algorithm}[th]
\caption{Approximate $\ell_{\infty}$ Lewis weights $\mathsf{ApproxLewis}(A)$}

\label{alg:approx-lewis}

\begin{algorithmic}[1]

\STATE\textbf{Input}:A symmetric polytope given by $-1_{n}\leq Ax\leq1_{n}$,
where $A\in\mathbb{R}^{n\times d}$

\STATE\textbf{Output}: Approximate $\ell_{\infty}$ Lewis weights
$w$ such that $w_{\text{true}}\leq w$ and $\sum w\leq d$

\STATE\textbf{Initialize}: $w_{i}^{\left(1\right)}=\frac{d}{n}$,
for $i=1,\dots,n$. $T=10\log n$.

\STATE\textbf{for} $k=1,\dots,T-1$

\STATE$\qquad W^{\left(k\right)}=\text{diag}\left(w^{\left(k\right)}\right)$.

\STATE$\qquad B^{\left(k\right)}=\sqrt{W^{\left(k\right)}}A$.

\STATE$\qquad$Let $S^{\left(k\right)}\in\mathbb{R}^{s\times n}$
be a random matrix where each entry is chosen i.i.d. from $N\left(0,1\right)$,
i.e. the standard normal distribution.

\STATE$\qquad$\textbf{for} $i=1,\dots,n$

\STATE$\qquad$$\qquad$$w_{i}^{\left(k+1\right)}=\frac{1}{s}\left\Vert S^{\left(k\right)}B^{\left(k\right)}\left(B^{\left(k\right)\top}B^{\left(k\right)}\right)^{-1}\left(\sqrt{w_{i}^{\left(k\right)}}a_{i}\right)\right\Vert _{2}^{2}$.

\STATE$w_{i}=\frac{1}{T}\sum_{k=1}^{T}w_{i}^{\left(k\right)}$ for
$i=1,\dots,m$.

\end{algorithmic}
\end{algorithm}

\section{Handling Affine Constraints\protect\label{sec:Handling-Affine-Constraints}}

In this section we show that the problem we solve is in full generality,
in the sense that even in the presence of affine constraints we can
still minimize objectives of the type \ref{eq:QSC-problem} without
significant overheads. Formally, given an objective of the form
\[
\min_{Nx=v}\sum_{i=1}^{n}f\left(\left(Ax-b\right)_{i}\right)
\]
where $A\in\mathbb{R}^{n\times d}$, and $N\in\mathbb{R}^{m\times d}$,
with $m<d$, we can minimize it to high precision using the algorithms
from Section \ref{sec:Generalization-to-QSC} with minimal changes.
Indeed, we can observe that, assuming the existence of an appropriate
residual solver, Algorithm \ref{alg:QSC-main} is completely unaffected
by this subspace constraint. Hence the only difficulty is posed by
solving the residual problem described in Algorithm \ref{alg:residual-solver}
while additionally enforcing the subspace constraint. Recall that
this problem, in its most general form, takes as input a matrix $A$,
as well as weight vectors $s,u$, and seeks an approximate minimizer
for
\[
\min_{\substack{x:\left\langle u,Ax\right\rangle =-1\\
Nx=0
}
}\left\langle s,\left(Ax\right)^{2}\right\rangle \,.
\]
To handle the affine constraint $Nx=0$, we can convert this objective
into an unconstrained problem via a change of variable. Indeed, let
$x_{0}$ be a point satisfying $Nx_{0}=v$. Then any $x$ in the affine
space can be expressed as 
\[
x=x_{0}+By\,,
\]
where $B\in\mathbb{R}^{d\times\dim\ker\left(C\right)}$ and $\text{im}\left(B\right)=\ker\left(N\right)$.
In other words, the columns of $B$ form a basis for the null space
of $N$. Thus our problem is equivalent to 
\[
\min_{\substack{y:\left\langle u,ABy\right\rangle =-1}
}\left\langle s,\left(ABy\right)^{2}\right\rangle \,,
\]
and takes exactly the form required for the $\ell_{\infty}$ regression
routine from Algorithm \ref{alg:residual-solver}, which leaves us
with an optimization problem over a lower dimensional space $y\in\mathbb{R}^{\dim\ker\left(N\right)}$.
Unfortunately, constructing the matrix $B$ explicitly may be costly,
so we want to avoid doing so. Instead, we claim that the regression
algorithm can be directly executed using solvers for matrices of the
type $A^{\top}DA$ and $N\left(A^{\top}DA\right)^{+}N^{\top}$ , where
$D$ is a positive diagonal.

The key difficulty lies in computing the least squares step
\begin{align}
x^{\left(t\right)} & =By^{\left(t\right)}\,,\label{eq:constrlsx}\\
y^{\left(t\right)} & =\arg\min_{y:\left\langle u,ABy\right\rangle =-1}\left\langle p^{(t)},\left(ABy\right)^{2}\right\rangle \,.\label{eq:costrlsy}
\end{align}
without explicitly accessing $B$. 
\begin{lem}
The solution for the least squares problem defined in (\ref{eq:constrlsx})
(\ref{eq:costrlsy}) can be explicitly computed as
\begin{align*}
x^{\left(t\right)} & =-\frac{1}{u^{\top}A\Delta}\cdot\Delta,\\
\Delta & =\left(A^{\top}PA\right)^{+}\left(-N^{\top}\left(N\left(A^{\top}PA\right)^{+}N^{\top}\right)^{+}N\left(A^{\top}PA\right)^{+}A^{\top}u+A^{\top}u\right)\,,
\end{align*}
where $P=\textbf{diag}\left(p^{\left(t\right)}\right)$
\end{lem}
\begin{proof}
We notice that the solution to this least squares problem has solution
\[
y^{\left(t\right)}=-\frac{1}{u^{\top}AB\left(B^{\top}A^{\top}PAB\right)^{+}B^{\top}A^{\top}u}\cdot\left(B^{\top}A^{\top}PAB\right)^{+}B^{\top}A^{\top}u\,,
\]
and equivalently $y^{\left(t\right)}$ satisfies $y^{\left(t\right)}=-\frac{1}{u^{\top}ABz}\cdot z$,
where
\begin{align*}
B^{\top}A^{\top}PABz & =B^{\top}A^{\top}u\,.
\end{align*}
Therefore there exists some $w=N^{\top}r$ such that:
\[
A^{\top}PA\cdot Bz-A^{\top}u=w\in\ker\left(B^{\top}\right)=\text{im}\left(B\right)^{\perp}=\ker\left(N\right)^{\perp}=\text{im\ensuremath{\left(N^{\top}\right)\,},}
\]
and thus we can write
\[
Bz=\left(A^{\top}PA\right)^{+}\left(w+A^{\top}u\right)\,,
\]
wich implies that
\[
N\left(A^{\top}PA\right)^{+}\left(w+A^{\top}u\right)=0\,,
\]
and hence
\[
N\left(A^{\top}PA\right)^{+}\left(N^{\top}r+A^{\top}u\right)=0\,.
\]
This allows us to explicitly express
\begin{align*}
r & =-\left(N\left(A^{\top}PA\right)^{+}N^{\top}\right)^{+}N\left(A^{\top}PA\right)^{+}A^{\top}u\,,\\
w & =-N^{\top}\left(N\left(A^{\top}PA\right)^{+}N^{\top}\right)^{+}N\left(A^{\top}PA\right)^{+}A^{\top}u\,,
\end{align*}
which finally yields 
\[
Bz=\left(A^{\top}PA\right)^{+}\left(-N^{\top}\left(N\left(A^{\top}PA\right)^{+}N^{\top}\right)^{+}N\left(A^{\top}PA\right)^{+}A^{\top}u+A^{\top}u\right)
\]
and 
\[
x^{\left(t\right)}=By^{\left(t\right)}=y^{\left(t\right)}=-\frac{1}{u^{\top}ABz}\cdot z\,.
\]
\end{proof}
Therefore the entire algorithm can be carried out as in the unconstrained
case.

\subsection{Lewis Weights in the Affine-constrained Setting}

In addition to being able to properly execute the least squared steps
in Algorithm \ref{alg:residual-solver}, one also requires a proper
initialization of weights. While in general, there is no direct notion
of Lewis weights in the case where affine subspace constraints are
included, we can apply the same reparametrization idea. After reparametrizing
the null space of $N$ in terms of the image of a matrix $B$, we
obtain an unconstrained problem involving the matrix $AB$, which
is the matrix for which we need to compute the $\ell_{\infty}$ Lewis
weight overestimates we use at initialization. Again, we argue this
can be done without explicit access to $B$. Note that the fixed point
iteration algorithm of \cite{cohen2019near} only requires computing
leverage scores of the underlying matrix $1_{i}^{\top}AB\left(B^{\top}A^{\top}PAB\right)^{+}B^{\top}A^{\top}1_{i}$.
Since, for increased efficiency, the algorithm also uses Johnson-Lindenstrauss
sketching, we further require being able to evaluate bilinear forms
of the type
\[
w^{\top}AB\left(B^{\top}A^{\top}PAB\right)^{+}B^{\top}A^{\top}u\,,
\]
where $P$ is a positive diagonal. As we saw before, we can evaluate
\begin{align*}
 & B\left(B^{\top}A^{\top}PAB\right)^{+}B^{\top}A^{\top}u\\
= & \left(A^{\top}PA\right)^{+}\left(-N^{\top}\left(N\left(A^{\top}PA\right)^{+}N^{\top}\right)^{+}N\left(A^{\top}PA\right)^{+}A^{\top}u+A^{\top}u\right)\,,
\end{align*}
which yields the expression for the required bilinear form, so it
can be directly applied inside the Lewis weights estimation algorithm. 

\end{document}